\documentclass[12pt,a4paper,leqno]{article}

\usepackage{amsmath,amscd,amsfonts,amsthm,amssymb,dsfont,bbm,manfnt}
\usepackage{hyperref,url,color,appendix,eucal,comment,tabu,float}
\usepackage[usenames,dvipsnames]{xcolor}
\usepackage[nottoc]{tocbibind}
\usepackage{multicol}
\usepackage[OT2,T1]{fontenc}
\usepackage[utf8]{inputenc}
\usepackage[all]{xy}
\usepackage{lscape,rotating} 
\usepackage{graphicx}
\usepackage[nolabel]{showlabels}
\usepackage{ulem}

\newenvironment{pf}
{\medskip\noindent {\it Proof.  }}
{\hfill\nobreak $\Box$ \par\bigbreak}

\newcommand{\isomo}{\overset{\sim}{\rightarrow}}

\newcommand{\tr}{\mathrm{tr}}

\newcommand{\SL}{\mathrm{SL}}
\newcommand{\ps}{\par \smallskip}

\newcommand{\Z}{\mathbb{Z}}
\newcommand{\Zp}{\Z_p}

\newcommand{\Q}{\mathbb{Q}}

\newcommand{\Qp}{\mathbb{Q}_p} 
\newcommand{\R}{\mathbb{R}}    
\newcommand{\C}{\mathbb{C}}

\newcommand{\AAA}{\mathbb{A}}

\newtheorem{thm}[subsection]{Theorem}
\newtheorem{lemma}[subsection]{Lemma}

\newtheorem{remark}[subsection]{Remark}
\newtheorem{cor}[subsection]{Corollary}

\newtheorem{prop}[subsection]{Proposition}
\newtheorem{example}[subsection]{Example}

\newtheorem{thmintro}{Theorem}

\newtheorem{thmapp}{Theorem}
\newtheorem{lemmapp}[thmapp]{Lemma}
\newtheorem{remapp}[thmapp]{Remark}
\newtheorem{propapp}[thmapp]{Proposition}
\newtheorem{definition}[subsection]{Definition}
\newtheorem{fact}[subsection]{Fact}

\usepackage{titlesec}

\titleformat{\section}{\bf \normalsize}{\arabic{section}.}{1 em}{}

\titleformat{\subsection}
{}
{\bf \arabic{section}.\arabic{subsection}.}
{0.5 em}
{\normalsize}

\titleformat{\subsubsection}[runin]
{\small \bf}
{}
{}
{}

\hypersetup{colorlinks=true,linkcolor=black,citecolor=black,urlcolor=Periwinkle,pdfborderstyle={/S/U/W 1}}

\makeatletter
\renewcommand\theequation{\arabic{section}.\arabic{equation}}
\@addtoreset{equation}{section}
\makeatother

\AtEndDocument{
 \bigskip{
   \footnotesize%
    \textsc{Ga\"etan Chenevier}, CNRS, DMA, \'Ecole Normale Sup\'erieure, 45 rue d'Ulm, 75230 Paris Cedex 5, France,\,\,
    \texttt{gaetan.chenevier@cnrs.fr} \par
    \addvspace{\medskipamount}
    \textsc{Olivier Ta\"ibi}, CNRS, UMPA, \'Ecole Normale Sup\'erieure de Lyon, Site Monod, 46 all\'ee d'Italie, 69364 Lyon Cedex 07, France,\,\, 
    \texttt{olivier.taibi@ens-lyon.fr}
  }
}

\begin{document}

\title{Statistics for Kneser $p$-neighbors}

\author{Ga\"etan Chenevier\thanks{During this work, the authors have been supported by the C.N.R.S. and by the project ANR-19-CE40-0015-02 COLOSS.}
\\ {\small (with an appendix by Olivier Ta\"ibi)} }


\maketitle

\begin{abstract} 
Let $L$ and $L'$ be two integral Euclidean lattices in the same genus. 
We give an asymptotic formula for the number of Kneser $p$-neighbors of $L$ which are isometric to $L'$,
when the prime $p$ goes to infinity. 
In the case $L$ is unimodular, 
and if we fix furthermore a subgroup $A \subset L$, 
we also give an asymptotic formula for the number of $p$-neighbors of $L$ containing $A$ and which are isomorphic to $L'$.
These statements explain numerical observations in \cite{unimodularhunting} and \cite{unimodularhunting2}.\par
In an Appendix, O. Ta\"ibi shows how to deduce from Arthur's results the existence of global parameters associated to automorphic representations of definite orthogonal groups over the rationals.  
\end{abstract}

{\footnotesize
\tableofcontents
}

\section{Introduction}

\subsection{The Minkowski-Siegel-Smith measure on a genus of integral lattices} Consider the standard Euclidean space $\R^n$, with inner product $x.y =\sum_i x_i y_i$. 
Recall that a lattice $L \subset \R^n$ is called {\it integral}
if we have $x.y \in \Z$ for all $x,y \in \Z$.
Such a lattice has a {\it genus} ${\rm Gen}(L)$, 
which is the set of all integral lattices $L'$ in $\R^n$ such that for all primes $p$
the symmetric bilinear spaces $L \otimes \Z_p$ and $L' \otimes \Z_p$, over the ring $\Z_p$ of $p$-adic integers, are isometric.
The Euclidean orthogonal group ${\rm O}(\R^n)$ naturally acts on ${\rm Gen}(L)$ with finitely many orbits, 
and we denote by
\begin{equation} \label{defXL} {\rm X}(L) = {\rm O}(\R^n) \backslash {\rm Gen}(L)\end{equation}  the set of isometry classes of lattices in ${\rm Gen}(L)$.\ps

Consider now the set $\mathcal{R}_n$ of {\it all} lattices in $\R^n$; this is a homogeneous space under ${\rm GL}_n(\R)$ 
with discrete stabilizers, so that $\mathcal{R}_n$ is a locally compact topological space in a natural way.
For any $L$ in $\mathcal{R}_n$, the orbit map ${\rm O}(\R^n) \rightarrow \mathcal{R}_n, g \mapsto g(L),$
is a topological covering with $|{\rm O}(L)|$-sheets onto its image, where 
${\rm O}(L) = \{ g \in {\rm O}(\R^n)\, \,|\, \, g(L)=L\}$ denotes
the isometry group of the lattice $L$, a compact discrete, hence finite, group.
Equip ${\rm O}(\R^n)$ with its volume $1$ Haar measure.
The orbit, or isometry class, of $L$ is thus endowed with an ${\rm O}(\R^n)$-invariant measure ${\rm m}$  
with total measure $${\rm m}({\rm O}(\R^n).L) = \frac{1}{|{\rm O}(L)|}.$$
The rational number $\frac{1}{|{\rm O}(L)|}$ will be called the {\it mass} of $L$. 
Assuming now $L$ is integral, and writing ${\rm Gen}(L)$ as a finite disjoint union of ${\rm O}(\R^n)$-orbits,
the construction above furnishes an ${\rm O}(\R^n)$-invariant measure ${\rm m}$ on 
${\rm Gen}(L)$, called the {\it Minkowski-Siegel-Smith} measure. 
A less artificial construction of the same measure can be obtained using a Haar measure 
on the adelic orthogonal group $G$ scheme of $L$ 
and the natural identification ${\rm Gen}(L) \isomo G(\Q) \backslash G(\AAA) / \prod_p G(\Z_p)$, 
with $\AAA$ the ad\`ele ring of $\Q$.
The {\it mass formula}, whose origin goes back to Smith, Minkowski and Siegel, gives a close and computable formula for
the rational number ${\rm m}({\rm Gen}(L))$. A  concrete (tautological) manifestation of these concepts is 
summarized in the following:\ps 

\begin{fact} \label{factarch}
In a genus $\mathcal{G}$ of integral Euclidean lattices, the probability of a random lattice to be isomorphic to $L$ is  $\frac{1/|{\rm O}(L)|}{{\rm m}(\mathcal{G})}$.
\end{fact}

\subsection{Kneser neighbors and their statistics}  Our first result will be a second manifestation of 
Fact \ref{factarch} in the setting of {\it Kneser neighbors}, that we first briefly review (see Sect. \ref{sect:pnei}).
We still assume that $L$ is an integral Euclidean lattice, and choose a prime $p$ not dividing $\det L$. 
An integral lattice $N  \subset \R^n$ is called a {\it $p$-neighbor} of $L$ if $L \cap N$ has index $p$ both in $L$ and in $N$. In this case, $L$ is a $p$-neighbor of $N$ as well.
This notion was introduced by Kneser originally in the case $p=2$.
As we will eventually let $p$ go to infinity, it will be harmless to rather assume $p$ is odd from now on. In this case\footnote{Indeed, for any such $p$-neighbor $N$ of $L$, 
we have $N [1/p]=L [1/p]$, and $N \otimes \Z_p \simeq L \otimes \Z_p$ by Lemma \ref{orbunimloc}.
} any $p$-neighbor of $L$ belongs to ${\rm Gen}(L)$.
We denote by $\mathcal{Nei}_p(L)$ the set of $p$-neighbors of $L$.
They are all easily described in terms of the set ${\rm C}_L(\Z/p)$ of isotropic $\Z/p$-lines of the non-degenerate symmetric bilinear space $L \otimes \Z/p$ over $\Z/p$ (a finite quadric). Indeed, the {\it line} map 
\begin{equation} \label{defell} \mathcal{\ell}:  \mathcal{Nei}_p(L) \longrightarrow {\rm C}_L(\Z/p), \, \, N \mapsto (pN+pL)/pL.\end{equation}
is {\it bijective}, and the lattice $N$ can be recovered from the isotropic line $l = \ell(N)$ by the following well-known construction.
For $l$ in ${\rm C}_L(\Z/p)$, set ${\rm M}_p(L; l) = \{ v \in L\, \, |\, \, v.l \equiv 0 \bmod p\}$ and choose $x \in L$ with 
$x.x \equiv 0 \bmod p^2$ whose image in $L \otimes \Z/p$ generates $l$. Such an $x$ always exists, and the lattice
\begin{equation} \label{voisdl} \,{\rm M}_p(L; l) \,+ \,\Z \,\frac{x}{p} \end{equation} does not depend on its choice: it is 
the unique $p$-neighbor $N$ of $L$ with $\ell(N)=l$, or equivalently, with $N \cap L = {\rm M}_p(L ; l)$. \ps

Since works of Kneser, Formula \eqref{voisdl} has been widely used as one of the main tools for constructing large quantities of new lattices in a fixed genus, from given ones and some primes. 
The cardinality of ${\rm C}_L(\Z/p)$ is easily computed: if $(\frac{a}{p})$ denotes the Legendre symbol of $a$ modulo $p$, we have for $n \geq 2$
{\small
\begin{equation}
\label{cardCLp}
|{\rm C}_L(\Z/p)| =  \left\{ \begin{array}{cc} 1+p+\cdots+p^{n-2} & \textrm{if $n$ is odd}, \\
1+p+\cdots+p^{n-2} + ( \frac{(-1)^{n/2} \det L}{p}) p^{n/2-1} & \textrm{if $n$ is even},
\end{array} \right.
\end{equation}
}\par
\noindent 
This cardinality only depends on $n$ and $\det L$, and for $n>2$ we have
 \begin{equation} \label{asymcl} |{\rm C}_L(\Z/p)| \sim p^{n-2}, \, \, \, \, {\rm for}\,\, p \rightarrow \infty. \end{equation}

An interesting question is to understand how many $p$-neighbors of a given $L$ lie in a given isometry class.
A first striking result in this direction, proved by Kneser as a consequence of his
{\it strong approximation theorem}, and originally applied by him to the case $p=2$ and the genera of unimodular lattices, 
is the following. Assume to simplify we have $n>2$ 
and that $\mathcal{G}={\rm Gen}(L)$ is a single {\it spinor genus} in the sense of Eichler 
(this is the most common situation). 
Consider the graph ${\rm K}_p(\mathcal{G})$ with set of vertices ${\rm X}(L)$ 
and with an arrow between $x$ and $y$ if $x$ has a $p$-neighbor isometric to $y$. 
Then ${\rm K}_p( \mathcal{G} )$ {\it is connected for any prime $p$ not dividing $\det L$}.
This theoretical result had decisive consequences for the classification of integral lattices of low
dimension or determinant, such as unimodular lattices (see {\it e.g.} \cite{scharlau} for a survey of such results). \ps

Another striking result, due to Hsia and J\"ochner \cite{HJ}, 
say still under the assumption $n>2$
and that $\mathcal{G}$ is a single spinor genus, asserts that for any 
$L'$ in ${\rm Gen}(L)$ there are infinitely many primes $p$ such that 
$L'$ is isomorphic to a $p$-neighbor of $L$.
(Their precise result is a bit more restrictive: see Corollary 4.1 {\it loc. cit.})
Our first main result is a more precise, quantitative, version of this result by Hsia and J\"ochner.\ps

\begin{thmintro}\label{thmstatgen} Let $L$ be an integral lattice in $\R^n$ with $n>2$. 
Assume $\mathcal{G}={\rm Gen}(L)$ is a single spinor genus, choose $L'$ in $\mathcal{G}$ and denote
by ${\rm N}_p(L,L')$ the number of $p$-neighbors of $L$ isometric to $L'$.
Then for $p \rightarrow \infty$ we have 
\begin{equation} \label{ramprop} \,\,\,\frac{ {\rm N}_p(L,L') }{| {\rm C}_L(\Z/p) |}  =  \frac{1/|{\rm O}(L')|}{{\rm m}(\mathcal{G})} + {\rm O}(\frac{1}{\sqrt{p}}),\end{equation} 
and we can even replace the $\frac{1}{\sqrt{p}}$ above by $\frac{1}{p}$ in the case $n>4$.
\end{thmintro}

In other words, for $L,L'$ in a same genus $\mathcal{G}$, and when the prime $p$ grows, 
the probability for a $p$-neighbors of $L$ to be isometric to $L'$ is proportional to the mass of $L'$. In particular, this probability does not depend on the choice of $L$ in $\mathcal{G}$. \ps

The most general version of Theorem~\ref{thmstatgen} that we prove 
actually applies to arbitrary genera, and even to genera of lattices equipped with level structures: see Theorem~\ref{thmmain}. 
For instance, if we do not assume that $\mathcal{G}={\rm Gen}(L)$ is a single spinor genus, then the same theorem holds 
with ${\rm m}(\mathcal{G})$ replaced by the common mass of the spinor genera $\mathcal{G}'$ in ${\rm Gen}(L)$, 
assuming we restrict in the asymptotics to primes $p$ belonging to a certain explicit union of arithmetic progressions (with Dirichlet density $\frac{{\rm m}(\mathcal{G}')}{{\rm m}(\mathcal{G})}$): 
see Remark~\ref{spinordirichlet}. In the case $n=2$, and for $p \nmid 2\det L$, then $L$ has 
either $0$ or $2$ $p$-neighbors, whose proper isometry classes are easily described using 
the theory of binary forms and class field theory: we won't say more about this fairly classical case. The case $n=1$ is irrelevant: we have ${\rm Gen}(L)=\{L\}$ and $L$ has no $p$-neighbor.\ps

\subsection{Remarks and proofs}
\label{subsect:rempfintro}

As far as we know, Theorem~\ref{thmstatgen} was only known before in the case $n=3$, 
and proved by Sch\"ulze-Pillot \cite[Corollary p. 120]{schulzepillot}. 
Our own inspiration comes from our work with Lannes \cite{chlannes}, 
in which we gave an exact formula for ${\rm N}_p(L,L')$ 
in the special case $L$ and $L'$ are even unimodular lattices of rank $\leq 24$. 
The asymptotics \eqref{ramprop} easily followed: see Scholium 10.2.3 {\it loc. cit.}\ps
 
Although it is unreasonable to ask for general exact formulas for ${\rm N}_p(L,L')$ 
in the case of arbitrary genera, it was quite natural to expect that 
Formula~\eqref{ramprop} should hold in general.
Indeed, after some elementary reductions, this formula is equivalent to a spectral gap property for 
the Hecke operator of Kneser $p$-neighbors acting on the space of all (automorphic!) 
functions ${\rm X}(L) \rightarrow \C$.
It is thus related to the generalized Ramanujan conjecture for automorphic representations of orthogonal groups,
in a way reminiscent of the classical constructions of Ramanujan graphs by Lubotzky-Phillips-Sarnak \cite{lps}, Margulis \cite{mar} or Mestre-Oesterl\'e \cite{mes} (isogeny graphs of supersingular elliptic curves). In the context above, the dominant eigenvalues correspond to spinor characters;
we review them in Sect. \ref{prelimspin} in the generality needed for the application to Theorem~\ref{secmainthmintro} below. \ps

As explained in \S \ref{sect:pf}, and quite naturally, the aforementioned spectral gap is 
a simple consequence of Arthur's endoscopic classification 
of automorphic representations of classical groups \cite{arthur} 
combined with estimates toward the generalized Ramanujan conjecture for general linear groups. 
In the end, the Jacquet-Shalika estimates for ${\rm GL}_m$ (any $m$) turn out to be enough, 
and we only need the full Ramanujan conjecture for ${\rm GL}_1$ (easy) 
and for classical holomorphic modular forms of weight $2$ for ${\rm GL}_2$ (Eichler, Shimura, Weil).
The basic results we need about Arthur's theory for the orthogonal group of positive definite rational quadratic forms 
were actually missing form the literature. 
In an appendix by Ta\"ibi, he explains how to deduce them from the results of Arthur \cite{arthur}, using similar arguments as in \cite{taibi_cpctmult}. \ps
Following a suggestion of Aurel Page, we also give in \S \ref{sect:pf} a second proof 
for the above spectral gap, in the case $n\geq 5$, which does not use Arthur's theory but instead
the (purely local) uniform estimates for matrix coefficients of unitary representations
given by Oh in \cite{oh}, in the spirit of the work of Clozel, Oh and Ullmo \cite{cou}
on equidistribution of Hecke orbits.  This second method, although less intuitive to us, 
is much less demanding in theory and leads to the same asymptotics. 
On the other hand, the first method makes it clear that the given asymptotics are optimal, 
and paves the way for the study of finer asymptotic expansions: we thus decided to give both arguments. 
For the sake of completeness, we also provide an alternative proof in the cases $n=3,4$ 
relying on the Jacquet-Langlands correspondence; 
for $n=4$ we even show that we may replace $\frac{1}{\sqrt{p}}$ by $\frac{1}{p}$ 
in the asymptotics provided $\det L$ is not a square (see Corollary~\ref{addendumdim4}).
\ps

As already said, the given asymptotics cannot be improved in general. 
For instance, if $L$ and $L'$ are Niemeier lattices with respective number of roots $24h$ and $24h'$, 
and if $\mathcal{G}_{24}$ denotes the genus of all Niemeier lattices,
we have {\small
\begin{equation}
\label{npllniemeier}
\frac{ {\rm N}_p(L,L') }{ |{\rm C}_L(\Z/p)|}  \,= \,\frac{ 1/|{\rm O}(L')| }{ {\rm m}(\mathcal{G}_{24}) }(\,1\,+\,  \frac{37092156523}{34673184000} \,(h-\frac{2730}{691})(h'-\frac{2730}{691})\, \frac{1}{p}) \,+\, {\rm O}(\frac{1}{p^2})
\end{equation}
}
\par \noindent by \cite[p.  317]{chlannes}. As explained {\it loc. cit.} the coefficient of $1/p$ in \eqref{npllniemeier} may be explained by studying the theta series of Niemeier lattices and a certain related eigenfunction on ${\rm O}(\R^{24})\backslash \mathcal{G}_{24}$ determined by Nebe and Venkov. This is a general phenomenon that could certainly be studied and developed further using higher genus theta series as in \cite{chlannes}, but that we do not pursue here.
Incidentally, we see in \eqref{npllniemeier} that the coefficient of $1/p$ can be positive or negative. 
In other examples, this coefficient can vanish as well. \ps
 
 \subsection{Biased statistics for unimodular lattices}
 \label{subsect:biasedintro} 
 
Theorem~\ref{thmstatgen}, applied to the genera of odd or even unimodular lattices, provide 
a theoretical explanation for some methods used in our forthcoming works \cite{unimodularhunting,unimodularhunting2},
in which we classify the isometry classes of unimodular lattices of rank $26, 27$ and $28$, giving
representatives for each of them as a cyclic $d$-neighbor of the standard lattice $\Z^n$.
In order to discover some unimodular lattices with very small mass, the rarest ones from the point of view of 
Theorem~\ref{thmstatgen}, we were led to {\it biase} the statistics by 
restricting to neighbors containing some fixed sublattice of rank $<n$. 
Our main motivation for writing this paper was to provide a mathematical 
justification of our experimental observations there, and to precisely compute the resulting biased probabilities.\ps

So we fix now an integral unimodular lattice $L \subset \R^n$, as well as a saturated subgroup $A \subset L$ (see \S \ref{sect:biased}).
To make statements as simple as possible in this introduction, we also assume $L$ is even.
For $p$ a prime, we are interested in the set $\mathcal{N}_p^A(L)$ of $p$-neighbors $N$ of $L$ 
with $A \subset N$. It is easily seen that for $p \, \nmid \, \det A$, the line map identifies 
$\mathcal{N}_p^A(L)$ with ${\rm C}_{A^\perp}(\Z/p)$, where $A^\perp$ denotes the orthogonal of $A$ in $L$ (see Lemma~\ref{linemapA}).
For $L'$ another unimodular integral lattice in $\R^n$, we 
are interested in the number ${\rm N}_p^A(L,L')$ of elements $N \in \mathcal{N}_p^A(L)$ which are isometric to $L'$.  \ps

For any integral lattice $U$, we denote by ${\rm emb}(A,U)$ the set of isometric embeddings $e : A \hookrightarrow U$ 
such that $e(A)$ is a saturated subgroup of $U$.  This is a finite set, whose cardinality only depends on the isomorphism class of $U$.
We set ${\rm m}_n^{\rm even}(A) \,=\, \sum_U\, |{\rm emb}(U,A)|/|{\rm O}(U)|$, 
the sum being over the isomorphism classes of rank $n$ even unimodular lattices $U$. 
We denote by ${\rm g}(G)$ the minimal number of generators of the finite abelian group $G$, 
and set ${\rm res}\, A\,= A^\sharp/A$ (see \S \ref{sect:notations} (ii)). 
We always have ${\rm rank}\, A \,+\, {\rm g}\,({\rm res}\, A) \leq n$.

\begin{thmintro} \label{secmainthmintro} Let $L$ be an even unimodular lattice of rank $n$ and $A$ a saturated 
subgroup of $L$. Assume ${\rm rank}\, A \,+\, {\rm g}({\rm res}\, A)\,< n-1$.
Then  
\begin{equation} 
\label{biasedstatint} 
\frac{{\rm N}_p^A(L,L')}{|{\rm C}_{A^\perp}(\Z/p)|} \,= \,\frac{ |{\rm emb}(A,L')|/|{\rm O}(L')|}{{\rm m}_n^{\rm even}(A)} \,+\, {\rm O}(\frac{1}{\sqrt{p}})\,\,
\, \,\,\,\,{\rm for}\, \,\,p \rightarrow \infty,
\end{equation}
and we can even replace the $\frac{1}{\sqrt{p}}$ above by $\frac{1}{p}$ in the case ${\rm rank}\, A \leq n-5$.
\end{thmintro}

Note that neither $|{\rm C}_{A^\perp}(\Z/p)|$, nor the right-hand side of \eqref{biasedstatint}, depend on the choice of $L$ containing $A$. We refer to Theorem~\ref{secmainthmodd} for a more general statement, which applies to odd unimodular lattices as well. Our proof is a combination of Theorem~\ref{thmmain} applied to the {\it inertial genus} of the lattice $A^\perp$, and of
the classical {\it glueing method}. 
Results of Nikulin \cite{nikulin}, namely the fact that the discriminant bilinear form of an Euclidean integral lattice determines its genus, as well as properties of the isometry groups of $p$-adic lattices, play a decisive role as well in the argument: see Sect. \ref{sect:biased}.\ps

We end this introduction by a simple illustration of Theorem \ref{secmainthmintro} 
to the genus $\mathcal{G}_{16}$ of even unimodular lattices in $\R^{16}$. 
As is well-known, it has two isometry classes, 
namely that of the lattices $L_1={\rm E}_{16}$ and $L_2={\rm E}_8 \oplus {\rm E}_8$, 
with respective root systems of type ${\bf D}_{16}$ and ${\bf E}_8 \coprod {\bf E}_8$. 
As an example, assume that $A$ is a root lattice with irreducible root system $R$ 
and having a saturated embedding in both $L_1$ and $L_2$: 
the possible $R$ are listed in the table below. 
The quantities $|{\rm emb}(A,L_i)|$ easily follow from Table 4 in \cite{king}.
\footnote{Note that the embeddings of ${\rm A}_7$ in ${\rm E}_8$ with orthogonal ${\rm A}_1$ 
do not have a saturated image.
} 
Theorem~\ref{secmainthmintro} applies, 
since we have ${\rm g}({\rm res}\, A) \leq 2$ in all cases, 
and shows, for $L \in \mathcal{G}_{16}$ and $p \rightarrow \infty$,
\begin{equation} 
\label{biasedX16}
\frac{{\rm N}_p^{A}(L,{\rm E}_8 \oplus {\rm E}_8)}{|{\rm C}_{A^\perp}(\Z/p)|}\, =\, \mu(R) \,+\, {\rm O}(\,\frac{1}{p}\,),
\end{equation}
where the constant $\mu(R)$ is given by the table below. The third column fits the fact that $L_1$ and $L_2$ have the same Siegel theta series of genus $1,2$ and $3$ (Witt, Igusa, Kneser).
Several other examples will be discussed in \cite{unimodularhunting}.

\begin{table}[H]
{\scriptsize
\renewcommand{\arraystretch}{1.8} \medskip
 \begin{center}
 \begin{tabular}{c|c|c|c|c|c|c|c|c|c|c}
$R$ & $\emptyset$ & ${\bf A}_r$ ($r\leq 3$) & ${\bf A}_4$ & ${\bf D}_4$ & ${\bf A}_5$ & ${\bf D}_5$ &  ${\bf A}_6$ & ${\bf D}_6$ & ${\bf A}_7$ & ${\bf D}_7$ \\
\hline   
$\mu(R)$ &  $286/691$ &  $286/691$  & $22/67$ & $22/31$ & $2/11$ & $22/31$ & $1/16$ & $2/5$ & $1/136$ & $2/17$ \\
\end{tabular} 
\label{tab:alphaR}
\end{center}
} 
\end{table}

\ps
{\sc Acknowledgements:} We are grateful to Olivier Ta\"ibi for several useful discussions and for writing the appendix to this paper. We are also grateful to Aurel Page for his remarks about the first version of this paper and for pointing out to us the work \cite{oh}.

\section{General notations and conventions}\label{sect:notations}

\ps\ps
In this paper, group actions will be on the left. We denote by $|X|$ the cardinality of the set $X$. 
For $n\geq 1$ an integer, we denote by $\Z/n$ the cyclic group $\Z/n\Z$ and 
by ${\rm S}_n$ the symmetric group on $\{1,\dots,n\}$. If $V$ if a finite dimensional vector-space
over a field $K$, and if $A$ is a subring of $K$, an $A$-{\it lattice} in $V$ is an $A$-submodule 
$L \subset V$ generated by a basis of $V$. 
\ps\ps

(i) ({\it Euclidean lattices})  If $V$ is an Euclidean space, we denote by $x \cdot y$ its inner product. A lattice in $V$ is a
$\Z$-lattice, or equivalently, a discrete subgroup $L$ with finite covolume, denoted ${\rm covol}\, L$. Its {\it dual lattice} is the lattice $L^\sharp$ defined as $\{ v \in V\, \, |\, \, v \cdot x \in \Z ,\,\, \, \forall x \in L\}$.  The lattice $L$ is {\it integral} if we have $L.L \subset \Z$, {\it i.e.} $L \subset L^\sharp$.  An integral lattice is called {\it even} if we have $x \cdot x \in 2\Z$ for all $x \in L$, and {\it odd} otherwise. The orthogonal group of $L$ is the finite group ${\rm O}(L)=\{ \gamma \in {\rm O}(V), \, \, \gamma(L)=L\}$.  \ps\ps

(ii) ({\it Residue}) Assume $L$ is an integral lattice in the Euclidean space $V$. The finite abelian group ${\rm res}\, L := L^\sharp/L$, called the {\it discriminant group} in \cite{nikulin} or the {\it glue group} in \cite{conwaysloane}, is equipped with a non-degenerate $\Q/\Z$-valued symmetric bilinear form ${\rm b}(x,y) = x.y \bmod \Z$. When $L$ is even, there is also a natural quadratic form ${\rm q} : {\rm res}\, L \rightarrow \Q/\Z$ defined by ${\rm q}(x) \equiv \frac{x.x}{2} \bmod \Z$, and which satisfies $ {\rm q}(x+y)={\rm q}(x)+{\rm q}(y)+{\rm b}(x,y)$ for all $x,y \in {\rm res}\, L$. We have $({\rm covol}\, L)^2 \,=\,  |{\rm res}\, L|$. This integer is also the {\it determinant} $\det L$ of the {\it Gram matrix} ${\rm Gram} (e) = (e_i \cdot e_j)_{1 \leq i,j \leq n}$ of any $\Z$-basis $e=(e_1,\dots,e_n)$ of $L$. \ps

(iii) ({\it Unimodular overlattices}) Fix $L \subset V$ as in (ii). A subgroup $I \subset {\rm res}\, L$ is called {\it isotropic}, if we have ${\rm b}(x,y) =0$ for all $x,y \in I$, and a {\it Lagrangian} if we have furthermore $|I|^2=|{\rm res}\, L|$. The map $M \mapsto M/L$ is a bijection between the set of integral lattices containing $L$ and the set of isotropic subgroups of ${\rm res}\, L$. In this bijection, $M/L$ is a Lagrangian if and only if $M$ is unimodular. In the case $L$ even, $M$ is even if and only if $I:=M/L$ is a {\it quadratic Lagrangian}, that is satisfies ${\rm q}(I)=0$. \ps

(iv) ({\it Localization}) Most considerations in (i), (ii) and (iii) have a local analogue, 
in which $V$ is replaced by a non-degenerate quadratic space $V_p$ (see below) 
over the field $\Q_p$ of $p$-adic numbers, 
and $L_p$ is a $\Z_p$-lattice in $V_p$.
Then $L_p^\sharp = \{ v \in L_p \, |\, v.L_p \subset \Z_p\}$ is a $\Z_p$-lattice, 
and when $L_p$ is integral, {\it i.e.} $L_p \subset L_p^\sharp$, 
we define ${\rm res}\, L_p = L_p^\sharp/L_p$ as above. 
This is a finite abelian $p$-group equipped with a natural non-degenerate $\Q_p/\Z_p$-valued symmetric bilinear form ${\rm b}_p$. 
In the case $p=2$ and ${\rm q}(L_2) \subset \Z_2$ ({\it i.e.} $L_2$ is {\it even}), 
${\rm res} \,L_2$ also has a natural $\Q_2/\Z_2$-valued quadratic form ${\rm q}_2$ inducing ${\rm b}_2$. 
We denote by ${\rm O}({\rm res}\, L_p)$ the automorphism group of the bilinear space $({\rm res}\, L_p,{\rm b}_p)$,
and set ${\rm I}(L_p):=\{ \gamma \in {\rm O}(L_p)\, |\, {\rm res}\, \gamma = 1\}$.
We have a natural exact sequence
\begin{equation}\label{resp} 1 \longrightarrow {\rm I}(L_p) \longrightarrow {\rm O}(L_p) \overset{{\rm res}}{\longrightarrow} {\rm O}({\rm res}\, L_p). \end{equation}
Assume now $V_p=L\otimes \Q_p$ for some integral Euclidean lattice $L$, and $L_p=L \otimes \Z_p$. The natural isomorphism $\Q/\Z \simeq \oplus_p \Q_p/\Z_p$ allows to identify 
${\rm res}\, L_p$ canonically with the $p$-Sylow subgroup of ${\rm res}\, L$, and $({\rm res}\, L,{\rm b})$ with the orthogonal direct sum of the $({\rm res}\, L_p,{\rm b}_p)$ : see {\it e.g.} \S 1.7 in \cite{nikulin} for more about these classical constructions.\ps

(v) ({\it Quadratic spaces}) A quadratic space over a commutative ring $k$ is a free $k$-module $V$ of finite rank equipped with a quadratic form, always denoted by ${\rm q}:  V \rightarrow k$. We often denote by $x.y={\rm q}(x+y)-{\rm q}(x)-{\rm q}(y)$ the associated symmetric $k$-bilinear form on $V$,
and say that $V$ is non-degenerate if this form is (so $\dim V$ is even if $k$ has characteristic $2$). 
We denote by $\det V$ the class in $k/k^{\times,2}$ of the determinant of any Gram matrix of $V$; 
so $V$ is non-degenerate if and only if $\det V \subset k^\times$.\ps

(vi) ({\it Orthogonal and Spin groups})
We denote respectively by ${\rm O}(V)$, ${\rm SO}(V)$ and ${\rm Spin}(V)$ the orthogonal group, the special orthogonal group and the spin group of the quadratic space $V$. Recall that we have ${\rm O}(V)=\{ g \in {\rm GL}(V)\, \,|\,\, {\rm q}\circ g = {\rm q} \}$, and unless $\dim V$ is even and $2$ is a zero divisor in $k$, a case that can mostly be ignored in this paper, the group ${\rm SO}(V)$ is defined as the kernel of the determinant morphism $\det : {\rm O}(V) \rightarrow \mu_2(k)$, with $\mu_2(k)=\{ \lambda \in k^\times \, \, |\, \, \lambda^2=1\}$ (see \cite[\S 5]{knuss} or \cite[\S 2.1]{chlannes}). 
By the Clifford algebra construction, we have a natural morphism $\mu : {\rm Spin}(V) \rightarrow {\rm SO}(V)$, and assuming $V$ is non-degenerate, 
a {\it spinor norm} morphism ${\rm sn} : {\rm O}(V) \rightarrow k^\times/k^{\times,2}$, such that the sequence 
\begin{equation}
\label{spinornorm} 
1 \longrightarrow \mu_2(k) \rightarrow {\rm Spin}(V) \overset{\mu}{\longrightarrow} {\rm SO}(V) \overset{{\rm sn}}{\longrightarrow} k^\times/k^{\times,2}
\end{equation}
is exact (see \cite[Thm. 6.2.6]{knuss}). Recall that if $s \in {\rm O}(V)$ denotes the orthogonal symmetry about an element $v \in V$ with ${\rm q}(v) \in k^\times$, we have ${\rm sn}(s) \equiv {\rm q}(v) \bmod k^{\times,2}$. It follows that ${\rm sn}$ is surjective when ${\rm q}(V)=k$, {\it e.g.} when $V$ contains a hyperbolic plane, or when $k$ is a field and $V$ is isotropic. 

\ps
(vii) ({\it Algebraic groups}) For any quadratic space $V$ over $k$ and any ring morphism $k \rightarrow k'$, we have a natural quadratic space $V \otimes k'$ obtained by scalar extension.  The groups ${\rm O}(V)$, ${\rm SO}(V)$ and ${\rm Spin}(V)$ are then the $k$-rational points of natural linear algebraic groups schemes over $k$ that we denote by ${\rm O}_V$, ${\rm SO}_V$ and ${\rm Spin}_V$. Also, the morphisms $\mu$, $\det$ and ${\rm sn}$ are functorial in $k$ whenever defined 
(in particular $\mu$ and $\det$ are group scheme morphisms). \ps

\section{Review and preliminaries on spinor characters}
\label{prelimspin}
  
	If $V$ is a quadratic space over $\Q$, we denote by $V_\infty$ the real quadratic space $V \otimes \R$, and for any prime $p$, we denote by $V_p$ the quadratic space $V \otimes \Q_p$ over the field $\Q_p$ of $p$-adic integers. 

\begin{lemma}\label{spinornormSOlemma} Assume $V$ is a non-degenerate quadratic space over $k$ with $k=\Q_p$ or $k=\Q$, and $\dim V \geq 3$. Then the spinor norm morphism ${\rm sn}$ in \eqref{spinornorm} is surjective, unless we have $k=\Q$ and $V_\infty$ is anisotropic, in which case its image is $\Q_{>0}/\Q^{\times,2}$. 
\end{lemma}

\begin{pf} See {\it e.g.} the statements 91:6 and 101:8 in \cite{omeara}. \end{pf}


\begin{lemma}\label{spinornormOlemma}  Assume $V$ is a non-degenerate quadratic space over $\Q_p$ with $\dim V \geq 3$. The morphism ${\rm sn} \times \det : {\rm O}(V) \longrightarrow \Q_p^\times/\Q_p^{\times,2} \times \{ \pm 1\}$ is surjective with kernel $\mu({\rm Spin}(V))$. When $V$ is isotropic, this kernel is furthermore the commutator subgroup of ${\rm O}(V)$. 
\end{lemma}

\begin{pf} The first assertion is a trivial consequence of Lemma \ref{spinornormSOlemma} and of the exactness of \eqref{spinornorm}.  The last assertion follows from \cite{omeara} \S 55:6a.
\end{pf}


A $\Z_p$-lattice $L$ in a quadratic space $V$ over $\Q_p$ (see \S \ref{sect:notations}) is called {\it unimodular} if we have ${\rm q}(L) \subset \Z_p$,
and if $(L,{\rm q}_{|L})$  a non-degenerate quadratic space over $\Z_p$, or equivalently, if $\det L \in \Z_p^\times/\Z_p^{\times,2}$. 
Note that for $p=2$, a unimodular $L$ is even and of even rank.

\begin{lemma}\label{integralspin}
Let $L$ be a unimodular $\Z_p$-lattice in the quadratic space $V$ over $\Q_p$.
Then ${\rm q}(L)=\Z_p$ and the sequence \eqref{spinornorm} induces an exact sequence 
$${\rm Spin}(L) \overset{\mu}{\longrightarrow} {\rm O}(L) \overset{{\rm sn} \times \det}{\longrightarrow} \Z_p^\times /\Z_p^{\times,2} \times \{ \pm 1\},$$
and the last arrow above is surjective for $\dim V \geq 2$.
\end{lemma}

\begin{pf} The existence and exactness of this sequence is a special case of \eqref{spinornorm}, applied to the quadratic space $L$ over $\Z_p$ (non-degenerate by assumption). The surjectivity assertion comes from the equality ${\rm q}(L)=\Z_p$,
which is a simple consequence of Hensel's lemma and ${\rm q}(L \otimes \Z/p)=\Z/p$, which in turns holds 
since $L \otimes \Z/p$ is non-degenerate of dimension $\geq 2$ over $\Z/p$.
\end{pf}

Fix $V$ a non-degenerate quadratic space over $\Q$.
We now recall a few facts about the adelic orthogonal group of $V$. 
It is convenient to fix a $\Z$-lattice $L$ in $V$ with ${\rm q}(L) \subset \Z$.
We may view $L$ as a (possibly degenerate) quadratic space over $\Z$,
and also  $L_p:=L \otimes \Z_p$ as a quadratic space over $\Z_p$; it is non-degenerate
(i.e. unimodular) for all but finitely many primes $p$. 
If $L'$ is another lattice in $V$, we have $L'_p=L_p$ for all but finitely many $p$.
Each ${\rm O}(V_p)$ is a locally compact topological group in a natural way, in which ${\rm O}(L_p)$
is a compact open subgroup.\ps

We denote by $\AAA_f$ the $\Q$-algebra of finite ad\`eles and set $\AAA = \R \times \AAA_f$.
As usual, we identify the group ${\rm O}_V(\AAA_f)$ with the group of sequences $(g_p)$ with $g_p \in {\rm O}(V_p)$ for all primes $p$, and $g_p \in {\rm O}(L_p)$ for all but finitely many $p$ (it does not depend on the choice of $L$). 
It contains ${\rm O}_V(\Q)={\rm O}(V)$ in a natural ``diagonal'' way.
Following Weil, ${\rm O}_V(\AAA_f)$ is a locally compact topological group if we choose as a basis of open neighborhoods of $1$ the subgroups of the form $\prod_p K_p$, with $K_p$ a compact open subgroup of ${\rm O}(V_p)$ equal to ${\rm O}(L_p)$ for all but finitely many $p$ (with their product topology). 
Similar descriptions and properties hold for adelic points of the algebraic $\Q$-groups ${\rm SO}_V$ and ${\rm Spin}_V$.
The morphism $\mu$ induces a topological group homomorphism ${\rm Spin}_V(\AAA_f) \rightarrow {\rm O}_V(\AAA_f)$.
The following lemma is a reformulation of Kneser's strong approximation theorem for Spin groups \cite{kneserstrong}.\ps

\begin{lemma}\label{strongapp} Let $V$ be a non-degenerate quadratic space over $\Q$ with $\dim V \geq 3$. Consider a map $\varphi : {\rm O}_V(\AAA_f) \rightarrow \C$ which is left-invariant under ${\rm O}(V)$, and right-invariant under some compact open subgroup of ${\rm O}_V(\AAA_f)$. Assume that $\varphi$ is right-invariant under $\mu({\rm Spin}(V_p))$ for some prime $p$ with $V_p$ is isotropic. 
Then $\varphi$ is left and right-invariant under $\mu({\rm Spin}(\AAA_f))$.
\end{lemma}

\begin{pf} Fix $x$ in ${\rm O}_V(\AAA_f)$ and define $\varphi_x : {\rm Spin}_V(\AAA_f) \rightarrow \C$ by $\varphi_x(h) = \varphi(\mu(h)x)$. Then $\varphi_x$ is left-invariant under ${\rm Spin}(V)$. For each prime $p$, $\mu({\rm Spin}(V_p))$ is a normal subgroup of ${\rm O}(V_p)$ by Lemma \ref{spinornormOlemma}. It follows that $\varphi_x$ is right-invariant under ${\rm Spin}(V_p)$ by the analogous assumption on $\varphi$. Assume $K$ is a compact open subgroup of ${\rm O}_V(\AAA_f)$ such that $\varphi$ is right $K$-invariant. By continuity of $\mu : {\rm Spin}_V(\AAA_f) \rightarrow {\rm O}_V(\AAA_f)$ there is a compact open subgroup $K'$ of ${\rm Spin}_V(\AAA_f)$ with $\mu(K') \subset xKx^{-1}$. Then $\varphi_x$ is right $K'$-invariant. 
By Kneser's strong approximation theorem \cite{kneserstrong} at the anisotropic place $p$, which applies as $\dim V \geq 3$, we have $${\rm Spin}_V(\AAA_f)= {\rm Spin}(V) \cdot K' \cdot {\rm Spin}(V_p),$$ so $\varphi_x$ is constant. We have proved $\varphi(\mu(h)x)=\varphi(x)$ for all $h \in {\rm Spin}(\AAA_f)$ and all $x\in {\rm O}_V(\AAA_f)$. We conclude as $x\mu(h) = x \mu(h) x^{-1} x$, and as $\mu({\rm Spin}(\AAA_f))$ is a normal subgroup of ${\rm O}_V(\AAA_f)$ by Lemmas \ref{spinornormOlemma} and \ref{integralspin}. 
\end{pf}


\begin{definition}\label{defspincar} Let $V$ be a non-degenerate quadratic space over $\Q$ and $K \subset {\rm O}_V(\AAA_f)$ a compact open subgroup. A {\it spinor character} of $V$ of genus $K$ is a group morphism 
$\sigma: {\rm O}_V(\AAA_f) \rightarrow \C^\times$ such that $\sigma(K)=\sigma({\rm O}(V))=1$. These characters form a group under multiplication that we denote by $\Sigma(K)$. 
\end{definition}

Fix $V$ and $K$ as above. For $\sigma \in \Sigma(K)$ and any prime $p$, we have a morphism $\sigma_p : {\rm O}(V_p) \rightarrow \C^\times$ defined by $\sigma_p(g)=\sigma(1 \times g)$. As $K$ contains $1 \times {\rm O}(L_p)$ for all $p$ big enough, we have $\sigma_p({\rm O}(L_p))=1$ for those $p$, and thus the product
$\sigma(g)=\prod_p \sigma_p(g_p)$ makes sense and holds for all $g=(g_p)$ in ${\rm O}_V(\AAA_f)$. 

\begin{lemma} \label{corstrappK}
If $\dim V \geq 3$ then any $\sigma \in \Sigma(K)$ is trivial on $\mu({\rm Spin}_V(\AAA_f))$.
\end{lemma}

\begin{pf} Fix $\sigma \in \Sigma(K)$. For all but finitely many primes $p$, the quadratic space $V_p$ is isotropic and so $\sigma_p$ is trivial on ${\rm Spin}(V_p)$ by the second assertion of Lemma \ref{spinornormOlemma}. We conclude by Lemma \ref{strongapp} applied to $\varphi=\sigma$.
\end{pf}

Spinor characters may be described more concretely as follows. Set 
\begin{equation}\label{defdelta} \Delta:=(\AAA_f^\times/\AAA_f^{\times,2}) \times \{ \pm 1\}^{\rm P} = \prod_{p}'\,\, (\Q_p^\times/\Q_p^{\times,2})  \times \{ \pm 1\}.\end{equation}
Here ${\rm P}$ denotes the set of primes, and
the restricted product above is taken with respect to the subgroups $\Z_p^\times/\Z_p^{\times,2}  \times 1$.
Definitely assume $\dim V \geq 3$.
By Lemmas \ref{spinornormOlemma} \& \ref{integralspin},  $\mu({\rm Spin}_V(\AAA_f))$ is a normal subgroup of ${\rm O}_V(\AAA_f)$
with abelian quotient and the map $\iota : = \prod_p {\rm sn} \times \det$ induces an isomorphism
\begin{equation} \label{idquotosp} \iota : {\rm O}_V(\AAA_f)/\mu({\rm Spin}_V(\AAA_f)) \isomo 
\Delta.\end{equation}
Fix a compact open subgroup $K \subset {\rm O}_V(\AAA)$. By Lemma \ref{corstrappK}, $\Sigma(K)$ coincides with the set of characters of the $2$-torsion abelian group $\Delta$ which are trivial on the 
subgroups $\iota({\rm O}(V))$ and $\iota(K)$. 
Assume $V$ is positive definite to fix ideas. By Lemma \ref{spinornormSOlemma},
$\iota({\rm O}(V)) \subset \Delta$ coincides then with the subgroup of ${\rm diag}(\lambda,\pm 1)$ with $\lambda \in \Q_{>0}$. We have thus an exact sequence
\begin{equation}
\label{exactDelta} 1 \rightarrow \Delta_1(K)  \rightarrow \Delta/\iota({\rm O}(V)\, K) \rightarrow  
\Delta_2(K) \rightarrow 1, \, \, \, {\rm with}
\end{equation}
{\small
\begin{equation}\label{defDelta12}
\Delta_1(K)=(\AAA_f^\times/ \AAA_f^{\times,2})/\langle \Q_{>0},\, {\rm sn}\, K^{\pm}\, \rangle
\,\,\,{\rm and}\,\,\,
\Delta_2(K)=\{ \pm 1\}^P/ \langle \pm 1, \det  K\, \rangle,
\end{equation}
}\par
\noindent where $K^{\pm}$ denotes the subgroup of $\gamma$ in $K$ with $\det \gamma = {\rm diag}(\pm 1)$ in $\{ \pm 1\}^{\rm P}$.
The group $\Delta_2(K)$ is usually easy to determine, whereas for $\Delta_1(K)$ it can be more tricky: see the discussion in \cite[\S 9 Chap. 15]{conwaysloane}. We shall content ourselves with the following classical observations.\ps

 Assume we have a finite set of primes $T$ such that
for all primes $p \notin T$, we have $K \supset 1 \times K'_p$ for some subgroup $K'_p \subset {\rm O}(V_p)$
satisfying $\det K'_p=\{\pm 1\}$ and $\Z_p^\times \subset {\rm sn}(K'_p \cap {\rm SO}(V_p))$.
Such a $T$ always exists, since theses conditions hold for $K'_p={\rm O}(L_p)$ 
when $L_p$ is unimodular, by Lemma \ref{integralspin}. 
Define ${\rm N}_T$ as the product of all odd primes in $T$, and of $8$ in the case $2 \in T$.
We have a natural morphism 
\begin{equation} \label{deltaT} 
\delta_T:  (\Z/{\rm N}_T)^\times / (\Z/{\rm N}_T)^{\times,2} \,=\,\prod_{p \in T} \Z_p^\times/\Z_p^{\times,2} \longrightarrow \Delta_1(K).
\end{equation}
The equality $\AAA_f^\times=\Q_{>0}\prod_p \Z_p^\times$ shows that $\delta_T$ is {\it surjective}. 
In particular, we have the well-known 
(see {\it e.g.} \S 102:8 in \cite{omeara} for the first assertion):

\begin{cor}\label{finitenessSK} $\Sigma(K)$ is a finite elementary abelian $2$-group.
Moreover,  if we may choose $T=\emptyset$ above, we have $\Sigma(K)=1$.
\end{cor}

\begin{pf} Note that $T=\emptyset$ implies ${\rm N}_T=1$, so $\Delta_1(K)=1$ by surjectivity of $\delta_T$, 
and also $\Delta_2(K)=1$ as we have $\det K = \{ \pm 1\}^P$.
\end{pf}

\section{Review of local $p$-neighbors and of the Hecke ring}
\label{sect:pnei}

Fix a prime $p$ and let $V$ be a non-degenerate quadratic space over $\Q_p$.
We denote by $\mathcal{U}(V)$ the set of unimodular $\Z_p$-lattices $L  \subset V$ (see \S \ref{prelimspin}).
We assume
\begin{equation} \label{existsunimodular} \mathcal{U}(V) \neq \emptyset\end{equation}
in all this section. We have a natural action of ${\rm O}(V)$ on $\mathcal{U}(V)$.
The isometry class of a non-degenerate quadratic space over $\Z_p$ is uniquely determined by its 
dimension and determinant (see e.g. \cite{omeara} 92:1a and \S 93).
As we have  $\det L \equiv \det V \, \bmod \Q_p^{\times, 2}$ for all $L$ in $\mathcal{U}(V)$, and $\Z_p^\times \cap \Q_p^{\times,2}=\Z_p^{\times,2}$, we deduce:

\begin{lemma} 
\label{orbunimloc}
$\mathcal{U}(V)$ form a single orbit under the action of ${\rm O}(V)$.
\end{lemma}

Recall that a $p$-neighbor of $L \in \mathcal{U}(V)$ is an element $N  \in \mathcal{U}(V)$ such that $L\cap N$ has index $p$ in $L$. 
We denote by $\mathcal{N}_p(L) \subset \mathcal{U}(V)$ the subset of $p$-neighbors of $L$. 
We have a natural {\it line} map 
\begin{equation} \label{lineneigh} \ell: \mathcal{N}_p(L) \longrightarrow {\rm C}_L(\Z/p), \, \, N \mapsto (pN+pL)/pL, \end{equation}
where ${\rm C}_L(\Z/p)$ denotes the set of isotropic lines in $L \otimes \Z/p$.
Indeed, note that we have $pN\cap pL = pM$, so that the image of $pN$ in $L/pL$ is naturally isomorphic to $pN/pM \simeq N/M$,
and satisfies ${\rm q}(pN) \equiv 0 \bmod p^2$. \ps
Note also that for $N \in \mathcal{N}_p(L)$ and $M:=N \cap L$, we have $L/M \simeq \Z/p$ and
$M . pN \equiv 0 \bmod p$,  so $M/pL$ is the orthogonal of $\ell(N)$ in $L \otimes \Z/p$.
Also, if $x \in pN$ generates $\ell(N)$, then we have ${\rm q}(x) \equiv 0 \bmod p^2$ and $N=\,M\,+\,\Z_p \, x/p$.
The following lemma is well-known.

\begin{lemma}\label{bijlinemap} The line map \eqref{lineneigh} is bijective. \end{lemma}

\begin{pf} Fix $\ell \in {\rm C}_L(\Z/p)$ and set $M= \{ v \in L \, \, |\, \, v. \ell \equiv 0 \bmod p\}$.
The finite quadratic space ${\rm res}\, M$ has order $p^2$ (see \S \ref{sect:notations} (iv)). 
It contains the isotropic subspace $L/M \simeq \Z/p$, and another $\Z/p$-line, generated by $x/p$
for any $x \in L$ generating $\ell$. One easily deduces that ${\rm res}\, M$ is a hyperbolic plane
of rank $2$ over $\Z/p$. It has thus exactly two isotropic lines, namely $L/M$ and another one $N/M$, 
where $N$ is the unique $p$-neighbor of $L$ containing $M$.
\end{pf}

%

\noindent We have natural actions of ${\rm O}(L)$ on $\mathcal{N}_p(L)$ and ${\rm C}_L(\Z/p)$,
and the line map is trivially ${\rm O}(L)$-equivariant. As ${\rm O}(L \otimes \Z/p)$ acts transitively 
on ${\rm C}_L(\Z/p)$, and since ${\rm O}(L) \rightarrow {\rm O}(L \otimes \Z/p)$ is surjective 
as $L$ is unimodular, we deduce:

\begin{lemma} 
\label{transpnei}
For $L$ in $\mathcal{U}(V)$, the group ${\rm O}(L)$ acts  transitively on $\mathcal{N}_p(L)$. 
\end{lemma}

Assume $V$ is isotropic, {\it e.g.} $\dim V \geq 3$.
Then we may choose $e,f \in L$ satisfying ${\rm q}(e)={\rm q}(f)=0$ and $e.f=1$. 
Then we have $L = (\Z_p \,e\, \oplus \,\Z_p\, f) \,\perp \,Q$ with $Q$ unimodular (and $\det Q \equiv - \det L$)
and
\begin{equation}\label{LNef}  N := (\Z_p \, pe \,\oplus \Z_p \,p^{-1}f\,) \,\perp\, Q\end{equation} 
clearly is a $p$-neighbor of $L$, with line $\ell(N)\,=\,\Z/p\, f$.
By Lemma \ref{transpnei}, any $p$-neighbor of $L$ has this form for a suitable choice of $e$ and $f$.

\begin{cor} \label{corsnpnei} If $g \in {\rm O}(V)$ satisfies $g(L) \in \mathcal{N}_p(L)$ then ${\rm sn}(g) \in p \Z_p^\times$.
\end{cor}

\begin{pf} By Lemma \ref{transpnei}, there is $h \in {\rm O}(L)$ such that $N:=hg(L)$ is given by Formula \eqref{LNef}.
Define $g' \in {\rm O}(V)$ by $g'(e)=pe$, $g'(f)=e/p$ and $g'_{|Q}={\rm id}_Q$. We have $g'(L)=N=hg(L)$ so $hg \in g'{\rm O}(L)$, 
${\rm sn}(g) \in {\rm sn}(g') \Z_p^\times$ (Lemma \ref{integralspin}) and we may assume $g=g'$.
We conclude as $g'$
 is the composition of the two reflexions about $e-pf$ and $e-f$, and we have ${\rm q}(e-pf){\rm q}(e-f)=(-p)(-1)=p$.
\end{pf}

Consider now the free abelian group $\Z\,\mathcal{U}(V)$ over the set $\mathcal{U}(V)$. 
This is a $\Z[{\rm O}(V)]$-module. 
The {\it Hecke ring} of $\mathcal{U}(V)$ is its endomorphism ring 
$${\rm H}_V = {\rm End}_{\Z[{\rm O}(V)]}(\Z\, \mathcal{U}(V) ).$$
The choice of $L \in \mathcal{U}(V)$ gives a natural isomorphism between ${\rm H}_V $ and the convolution ring
of compactly supported functions $f : {\rm O}(V) \rightarrow \Z$ such that 
$f(kgk')=f(g)$ for all $g \in {\rm O}(V)$ and all $k,k' \in {\rm O}(L)$. 
Concretely, an element $T \in {\rm H}_V $ corresponds to the function $f$ defined by $TL \,=\, \sum \,f(g)\, gL$,
the (finite) sum being over the elements $g$ in ${\rm O}(V)/{\rm O}(L)$. 

\begin{definition} The {\rm $p$-neighbors operator} is the element ${\rm T}_p$ of ${\rm End}(\Z\,\mathcal{U}(V))$ defined by
${\rm T}_p L \, =\, \sum_{N \in \mathcal{N}_p(L)} \,N$. We clearly have ${\rm T}_p \in {\rm H}_V $.
\end{definition}

We know since Satake \cite{satake} that ${\rm H}_V $ is a commutative finitely generated ring. 
One simple reason for this commutativity is that the natural anti-involution $T \mapsto T^{\rm t}$, 
induced by transpose of correspondences (see {\it e.g.} \cite[\S 4.2.1]{chlannes}),
is the identity of ${\rm H}_V$. 
This in turn follows from the fact that for any $L$ and $L'$ in $\mathcal{U}(V)$ we have an isomorphism
of abelian groups $L/(L\cap L') \simeq L'/(L \cap L')$  (use either the Cartan decomposition given by \cite[\S 9.1]{satake} or the direct argument given in \cite[Prop. 3.1.1 \& Prop. 4.2.8]{chlannes}).\ps

We denote by ${\rm Hom}({\rm H}_V,\C)$ the set of ring homomorphisms ${\rm H}_V \rightarrow \C$, 
also called {\it systems of Hecke eigenvalues}. 
There is a classical bijection $\chi \mapsto U(\chi)$ between ${\rm Hom}({\rm H}_V,\C)$ and the set of 
isomorphism classes of irreducible {\it unramified} $\C[{\rm O}(V)]$-modules $U$, {\it i.e.} such that 
the fixed subspace $U^{{\rm O}(L)}$ is nonzero for some
(hence all) $L \in \mathcal{U}(V)$. In this bijection, $U(\chi)^{{\rm O}(L)}$ is a line equipped with a natural action 
of ${\rm H}_V$, and it gives rise to the morphism $\chi$.\ps\ps

A second interpretation of ${\rm Hom}({\rm H}_V,\C)$ is given by the {\it Satake isomorphism}.
In Langlands's interpretation, this isomorphism takes the following form.
Consider\footnote{In the case $\dim V$ even, we conveniently make here a slight entorse to Langlands's widely used conventions, hoping the forgiveness of the knowledgable reader.} the complex algebraic group $\widehat{{\rm O}_V}$ defined by 
$\widehat{{\rm O}_V}={\rm Sp}_{2r}(\C)$ (standard complex symplectic group of rank $k$)
if $\dim V = 2r+1$ is odd, and by $\widehat{{\rm O}_V}={\rm O}_{2r}(\C)$ (standard complex orthogonal group of rank $k$) if $\dim V=2r$ is even. In all cases, we define ${\rm n}_V:=2r$ and have a natural, called {\it standard}, complex representation 
$${\rm St} : \widehat{{\rm O}_V} \rightarrow {\rm GL}_{{\rm n}_V}(\C).$$ 
The Satake isomorphism \cite{satake,cartier,grossatake}  induces a canonical bijection $\chi \mapsto {\rm c}(\chi)$, between ${\rm Hom}({\rm H}_V, \C)$
and the set of semi-simple conjugacy classes $c \subset \widehat{{\rm O}_V}$ 
satisfying furthermore $\det {\rm St}(c) =-1$ in the case $\dim V=2r$ is even and $(-1)^r \det V$
is not a square in $\Q_p^\times$, and $\det {\rm St}(c) =1$ otherwise. \ps

\begin{remark} \label{heckeSOV}
{\rm  (${\rm SO}$ versus ${\rm O}$)
For $L$ in $\mathcal{U}(V)$, the inclusion ${\rm O}(L) \subset {\rm O}(V)$ induces
an isomorphism ${\rm O}(L)/{\rm SO}(L) \simeq {\rm O}(V)/{\rm SO}(V)$ of groups of order $2$. 
It follows that ${\rm SO}(V)$ acts transitively on $\mathcal{U}(V)$. Define 
$${\rm H}'_V = {\rm End}_{\Z[{\rm SO}(V)]}(\Z\, \mathcal{U}(V)).$$
Again, the choice of an $L \in \mathcal{U}(V)$ identifies the ring ${\rm H}'_V$ with 
the convolution ring of compactly supported functions ${\rm SO}(L)\backslash {\rm SO}(V)/{\rm SO}(L) \rightarrow \Z$. 
Also, ${\rm H}'_V$ has a natural action of $\Z/2={\rm O}(V)/{\rm SO}(V)$, with fixed subring ${\rm H}_V$.
As observed by Satake, we have ${\rm H}'_V={\rm H}_V$ unless
 $\dim V=2r$ is even and $V$ has Witt index $r$.
The dual description given above of the Satake isomorphism for ${\rm H}_V$ follows 
then from that for the Hecke ring ${\rm H}'_V$ associated to the {\it connected} semisimple group scheme ${\rm SO}_L$ over $\Z_p$ for $L \in \mathcal{U}(V)$. See Scholium 6.2.4 in \cite{chlannes} for a discussion of the case  $\dim V=2r$ and $V$ has Witt index $r$.
}
\end{remark}

\begin{lemma} \label{stakeTp}  For any ring homomorphism $\chi : {\rm H}_V \rightarrow \C$ we have 
$$\chi({\rm T}_p) \,=\, p^{\frac{\dim V}{2}-1} \,{\rm Trace} \,{\rm St}({\rm c}(\chi)).$$
\end{lemma}

\begin{pf}
If $V$ has Witt index $[\dim V/2]$,
the lemma is a consequence of Formula~\eqref{LNef} and of Gross's 
exposition of the Satake isomorphism for split reductive groups over $\Q_p$ in \cite{grossatake}: 
see \S 6.2.8 in \cite{chlannes}, especially Formula (6.2.5) in the case $\dim V$ even, 
and the formula in the middle of page 160 in the case $\dim V$ odd.
Assume now $\dim V=2r$ is even and $V$ has Witt index $r-1$. 
We provide a proof in this case as we are not aware of any reference. \ps
{\scriptsize
Note that we have ${\rm H}_V={\rm H}'_V$, so we may use the description of the Satake isomorphism
given in Sect. II \S 6 \& 7 of \cite{borel} and Sect. III and IV of \cite{cartier}.
For $r=1$ we have ${\rm T}_p=0={\rm Trace} \,{\rm St}({\rm c})$ for all $c \in \widehat{{\rm O}_V}$, 
so we may assume $r\geq 2$. 
Fix $L \in \mathcal{U}(V)$. By assumption on $V$, we can find 
elements $e_i$ and $f_i$ of $L$, $1\leq i \leq r-1$, with $e_i.e_j=f_i.f_j=0$ and $e_i.f_j=\delta_{i,j}$
for all $i,j$. We have $L = L_0 \perp \bigoplus_{1\leq i \leq r-1}^{\perp} (\Z_p e_i \oplus \Z_p f_j)$,
with $L_0$ unimodular, anisotropic, of rank $2$ over $\Z_p$. 
The subgroup $T$ of ${\rm SO}(V)$ stabilizing each line $\Q_p e_i$ and $\Q_p f_j$
is a maximal torus, and its subgroup $A$ acting trivially on $V_0=L_0 \otimes \Q_p$ is a maximal split torus.
The rational characters\footnote{
We denote as usual by ${\rm X}^\ast(S)$ and ${\rm X}_\ast(S)$ respectively 
the groups of characters and cocharacters of a torus $S$, 
and by $\langle-,-\rangle$ the perfect pairing 
${\rm X}_\ast(S) \otimes {\rm X}^\ast(S) \rightarrow \Z$.} 
$\epsilon_i \in {\rm X}^\ast(A)$, defined for $i=1,\dots,r-1$ by $te_i= \epsilon_i(t) e_i$, form a $\Z$-basis of ${\rm X}^\ast(A)$; 
we denote by $\epsilon_i^\ast \in {\rm X}_\ast(A)$ its dual $\Z$-basis.
The subgroup of ${\rm SO}(V)$ preserving $\sum_{j \leq i} \Q_p e_j$ for all $i$ is a Borel subgroup $B$, 
with unipotent radical denoted by $N$.
The (positive) roots of $A$ in ${\rm Lie} \,B$ are the $\epsilon_i \pm \epsilon_j$ with $i <j$ (with root space of dimension $1$), and the $\epsilon_i$ (with root space of dimension $2$), so that the root system of $A$ is of type ${\bf B}_{r-1}$. We denote by $W$ its Weyl group.
The Satake isomorphism is an isomorphism
\begin{equation} \label{satakecrude} {\rm Sat}: {\rm H}_V \otimes \C \isomo \C[{\rm X}_\ast(A)]^{\rm W}. \end{equation} 
An element $\lambda = \sum_i \lambda_i \epsilon_i^\ast$ in ${\rm X}_\ast(A)$ is $B$-dominant
if we have $\lambda_1 \geq \lambda_2 \geq \cdots \geq \lambda_{r-1} \geq 0$.
For each such $\lambda$ we denote by $\chi_\lambda \in \Z[{\rm X}_\ast(A)]^{\rm W}$ 
the sum of the elements in the $W$-orbit of $\lambda$, and by 
${\rm T}_\lambda \in {\rm H}'_V$ the element associated to the characteristic function 
of ${\rm SO}(L) \lambda(p) {\rm SO}(L)$. 
By \cite[p. 53]{satake}, we have ${\rm Sat}({\rm T}_\lambda) \,= \,\sum_\mu \,{\rm c}_{\lambda; \mu}\, \chi_\mu$, 
where $\mu$ runs among the dominant weights with $(\mu_i) \leq (\lambda_i)$ for the lexicographic ordering $\leq$ on $\Z^{r-1}$,
for some constants ${\rm c}_{\lambda; \mu}$.  
By the remark at the top of p. 54 {\it loc. cit.}, we have ${\rm c}_{\lambda; \lambda}=\delta(\lambda(p))^{-1/2}$, 
with $\delta$ the modulus character of $B$.
Set $\lambda=\epsilon_1^\ast$. We have ${\rm T}_\lambda={\rm T}_p$ by Lemma \ref{LNef}. 
Using the root space decomposition recalled above, we see $\delta(\lambda(p))^{-1/2}=p^{r-1}$.
Also, for a dominant $\mu$ we have $\mu \leq \lambda$ if, and only if, $\mu=0$ or $\mu=\lambda$. 
On the other hand, we have ${\rm SO}(L)\lambda(p){\rm SO}(L)\, \cap \,N \,{\rm SO}(L)=\emptyset$
since the spinor norm of ${\rm O}(V)$ is trivial on the divisible group $N$, and is $p$ on $\lambda(p)$.
By definition of {\rm Sat}, this shows ${\rm c}_{\lambda;0}=0$ and
\begin{equation} \label{sattpinterm}  {\rm Sat}({\rm T}_p) \, = \, p^{r-1} \, \chi_{\epsilon_1^\ast}.\end{equation}
On the dual side, we also fix a $\C$-basis $e'_1,\dots,e'_r,f'_1,\dots,f'_r$ of the quadratic space $\C^{2r}$ with
$e'_i.e'_j=f'_i.f'_j=0$ and $e'_i\cdot f'_j=\delta_{i,j}$ for all $i,j$, which defines a maximal torus 
$\widehat{T}$ and a Borel subgroup $\widehat{B}$ of ${\rm SO}_{2r}(\C)$, as well as an order $2$ element
$\sigma \in {\rm O}_{2r}(\C)$ fixing $e'_i$ and $f'_j$ for each $i,j<r$, and exchanging $e'_r$ and $f'_r$.
We have a $\Z$-basis $\epsilon'_i$ of ${\rm X}^\ast(\widehat{T})$ defined by $t.e'_i=\epsilon'_i(t)e'_i$.
By definition of the Langlands dual, we have a given isomorphism $\iota : {\rm X}_\ast(T) \isomo {\rm X}^\ast(\widehat{T})$
sending $\epsilon_i^\ast$ to $\epsilon'_i$ for $i=1,\dots,r-1$. Also, the subgroup $U \subset \widehat{T}$ 
of elements of the form $\sigma t \sigma^{-1} t^{-1}$ with $t \in \widehat{T}$ coincides with $\cap_{1 \leq i \leq r-1}{\rm ker}\, \epsilon'_i$,
and if we set $\widehat{A}=\widehat{T}/U$ then $\iota$ defines a isomorphism ${\rm X}_\ast(A) \simeq {\rm X}^\ast(\widehat{A})$. It is a straightforward exercise to check that any semisimple element $g \in {\rm O}_{2r}(\C)$ with $\det g=-1$ is ${\rm O}_{2r}(\C)$-conjugate to an element of the subset $\widehat{T} \sigma$. Moreover, two semisimple elements of ${\rm O}_{2r}(\C)$ are conjugate if, and only if, they have the same characteristic polynomial. For $t\sigma \in \widehat{T}\sigma$, this characteristic polynomial  is 
\begin{equation} \label{charpoly} {\rm det} ( X - {\rm St}(t\sigma) ) = (X-1)(X+1) \prod_{i=1}^{r-1}(X - \epsilon'_i(t))(X - \epsilon'_i(t)^{-1}).\end{equation}
It follows that for $t,t' \in \widehat{T}$ we have an equivalence between: (i) $t \sigma$ and $t' \sigma$ are 
conjugate in ${\rm O}_{2r}(\C)$, and (ii) the images of $t$ and $t'$ in $\widehat{A}$ are conjugate under $W$ (which acts naturally on ${\rm X}^\ast(\widehat{A})$ hence on $\widehat{A}$). As a consequence, the inclusion $\mu : \widehat{T}\sigma \rightarrow {\rm O}_{2r}(\C)^{\det =-1}$ and the projection $\nu : \widehat{T} \sigma \rightarrow \widehat{A}, t\sigma \mapsto tU,$ define a bijection 
\begin{equation} \label{specsat} \mu \circ \nu^{-1} : \widehat{A}/W \isomo {\rm O}_{2r}(\C)^{\det =-1}/{\rm O}_{2r}(\C).\end{equation}
The bijection $\chi \mapsto {\rm c}(\chi)$ is defined by composing \eqref{specsat}, together with \eqref{satakecrude} and the natural bijection between the sets
${\rm Hom}(\C[{\rm X}_\ast(A)]^{\rm W},\C)$ and $\widehat{A}/W$ defined by $\iota$, 
In particular, if $f : {\rm O}_{2r}(\C) \rightarrow \C$ denotes a polynomial class function, then the map 
$t \mapsto f(t\sigma)$ factors through a $W$-invariant polynomial map $\widehat{A} \rightarrow \C$, 
or equivalently using $\iota$, through an element of $\C[{\rm X}_\ast(A)]^W$ that we denote by $f_\ast$.
In the case $f$ is the trace of the strandard representation ${\rm St}$ of ${\rm O}_{2r}(\C)$, we see
$$f_\ast = \chi_{\epsilon_1^\ast},$$
and conclude by \eqref{sattpinterm}.
}
\end{pf}

We end this discussion by recalling the Satake parameters of the $1$-dimen\-sional unramified 
$\C[{\rm O}(V)]$-modules, or equivalently, of the group homomorphisms $\mu : {\rm O}(V) \rightarrow \C^\times$
with $\mu({\rm O}(L))=1$ for some (hence all) $L \in \mathcal{U}(V)$. 
Assume $\dim V \geq 3$. By Lemmas \ref{spinornormOlemma} \& \ref{integralspin}, there are exactly 
$2$ such characters, namely the trivial character $1$ and the character $\eta$ defined by
$$\eta(g) = (-1)^{{\rm v}_p( {\rm sn}(g))}$$
where ${\rm v}_p(x)$ denotes the $p$-adic valuation of the element $x \in \Q_p^\times$. 
Recall ${\rm n}_V=2r$. We denote by $\Delta \subset \widehat{{\rm O}_V}$ the unique semi-simple conjugacy class such that the eigenvalues of ${\rm St}(\Delta)$ are:\ps

-- the $2r$ positive square roots $p^{ \pm \frac{2i-1}{2} }$ with $i=1,\dots,r$ for $\dim V$ odd, \ps
-- the $2r-1$ integers $p^i$ for $|i|\leq r-1$, and the element $1$ (resp. $-1$) if $\dim V=2r$ is even and $V$ has Witt index $r$ (resp. $r-1$).\ps


\begin{lemma}\label{sataketrivialeta} 
Assume $\dim V \geq 3$ and that $\chi : {\rm H}_V \rightarrow \C$ satisfies ${\rm c}(\chi)=\Delta$ {\rm (}resp. 
${\rm c}(\chi)=-\Delta${\rm )}. Then we have ${\rm U}(\chi) \simeq 1$ {\rm (}resp. ${\rm U}(\chi) \simeq \eta${\rm )}. 
\end{lemma}
{\scriptsize
\begin{pf} This is well-known but we provide a proof as we could not find unsophisticated references.
To fix ideas we assume $\dim V=2r$ is even and $V$ has Witt index $r-1$ (the most complicated case), and use the notations introduced during the proof of Lemma~\ref{stakeTp}.
To any unramified character $\nu :  T \rightarrow \C^\times$ (see \S 3.2 in \cite{cartier}), and to the Borel subgroup $B=TN$, 
may associate an induced representation denoted by $I(\nu)$ {\it loc. cit.} \S 3.3. This representation contains a unique unramified subquotient, necessarily of the form ${\rm U}(\chi)$ for some unique $\chi \in {\rm Hom}({\rm H}_V,\C)$. Following \cite{cartier} \S 4.3, the elements $\nu$ and ${\rm c}(\chi)$ are related as follows. The character $\nu$
defines by restriction an unramified character of the split torus $A$, hence factors through the 
surjection ${\rm ord}_A : A\rightarrow {\rm X}_\ast(A)$ defined by Formula (S) p. 134 {\it loc. cit.}
It may thus be viewed as an element of ${\rm Hom}({\rm X}_\ast(A),\C^\times)$, or equivalently of $\widehat{A}$ using the dual interpretation: 
the $W$-orbit of this element is ${\rm c}(\chi)$ by Formula (35) {\it loc. cit.} and  Formula \eqref{specsat}.
The trivial representation $1$ is obviously a subspace of the induced representation $I(\nu)$ with $\nu:=\delta^{-1/2}$ and $\delta$ the modulus character of $B$. As a consequence, the character $\eta$ is a subspace of $I(\nu)\otimes \eta = I( \nu \eta)$
where $\eta$ is viewed as a character of $T$ by restriction. A trivial computation using the description of root spaces of $A$ given in the proof of Lemma~\ref{stakeTp} shows $\nu(\epsilon_i^\ast(p))=p^{r-i}$ for $i=1,\dots,r-1$.
On the other hand, we have $\eta (\epsilon_i^\ast(p)) = -1$ for all $i=1,\dots r-1$ by Lemma \ref{corsnpnei}. 
As a consequence, by \eqref{charpoly} the $2r$ eigenvalues of the element ${\rm St}({\rm c}(\chi))$ are $1,-1$, together with the $2r-2$ elements 
$\epsilon p^{\pm i}$ with $i=1,\dots,r$, and $\epsilon=1$ in the case ${\rm U}(\chi)\simeq 1$, and $\epsilon=-1$ in the case ${\rm U}(\chi) \simeq \eta$. 
\end{pf}
}

\section{Setting and statement of the main theorem}

\subsection{Class sets of lattices with level structures}\label{classsets}

Let $V$ be a non-degenerate quadratic space over $\Q$.
We fix a compact open subgroup $K$ of ${\rm O}_V(\AAA_f)$ and consider the sets 
\begin{equation}\label{defgenKXK} {\rm Gen}(K) = {\rm O}_V(\AAA_f) / K\, \, \,\,\, {\rm and}\, \, \,\, {\rm X}(K) = {\rm O}(V) \backslash {\rm Gen}(K).
\end{equation}
For any finite set $S$ of primes, we have rings $\Q_S = \prod_{p \in S} \Q_p$ and $\Z_S = \prod_{p \in S} \Z_p$ and a quadratic space $V_S = \prod_{p \in S} V_p$ over $\Q_S$.
We may and do fix such an $S$ such that we have $K=K_S \times \prod_{p \notin S} K_p$,
with compact open subgroups $K_S \subset {\rm O}(V_S)$ and $K_p \subset {\rm O}(V_p)$ for $p \notin S$.
We also set 
\begin{equation}\label{defgenkp} {\rm Gen}(K_p)={\rm O}(V_p)/K_p\end{equation} 
for each prime $p \notin S$, and ${\rm Gen}(K_S)={\rm O}(V_S)/K_S$.
Both ${\rm Gen}(K)$, ${\rm Gen}(K_S)$ and the ${\rm Gen}(K_p)$ are pointed sets (with distinguished points $K, K_S$ and the $K_p$), and we have a natural restricted product decomposition of pointed sets
\begin{equation}\label{locdecGen}  {\rm Gen}(K) = {\rm Gen}(K_S) \times \prod'_{p \notin S}  {\rm Gen}(K_p). \end{equation}
\par \noindent
Before going further, we recall how to interpret these sets in concrete cases.\ps

\begin{example}  \label{gengenus} {\rm (}Lattice case{\rm)}
{\rm
Recall, following Eichler and Weyl, 
that there is an action of the group ${\rm O}_V(\AAA_f)$ 
on the set of $\Z$-lattices $L$ in $V$, say $(g,L) \mapsto g(L)$, 
uniquely determined by the property  $g(L)_p=g_p(L_p)$ for each prime $p$. 
Fix $L \subset V$ a $\Z$-lattice\footnote{
We only have ${\rm q}(L) \subset \frac{1}{N}\Z$ for some $N \geq 1$ a priori, but
we may even assume ${\rm q}(L) \subset \Z$ if we like,
replacing $L$ with $NL$.}, 
set $K= \prod_p K_p$ with $K_p={\rm O}(V_p)$, and take $S=\emptyset$. 
Then the map $g \mapsto g(L)$ defines an ${\rm O}(V)$-equivariant bijection 
\begin{equation}\label{GenKGenL}  {\rm Gen}(K) \isomo {\rm Gen}_\Q(L) \end{equation}
where ${\rm Gen}_\Q(L)$ denotes the genus of $L$ in $V$.
This bijection identifies ${\rm X}(K)$ with the set of isometry classes of lattices in this genus. 
For each prime $p$, we also define ${\rm Gen}(L_p)$ as the set of $\Z_p$-sublattices of $V_p$ isometric to $L_p$. Both ${\rm Gen}_\Q(L)$ and the ${\rm Gen}(L_p)$ have a distinguished point, namely $L$ and the $L_p$, and the decomposition \eqref{locdecGen} identifies ${\rm Gen}(K)$ componentwise with ${\rm Gen}_\Q(L) = \prod'_p {\rm Gen}(L_p)$
under $L' \mapsto (L'_p)_p$.  
}
\end{example}

\begin{example}  \label{gengenuslevel}  {\rm (}Level structure case{\rm)}
{\rm
Fix $L \subset V$ a lattice, a finite set $S$ of primes, choose a compact open subgroup 
$K_S \subset {\rm O}(L_S)$, with $L_S = \prod_{p \in S} L_p$, and set $K= K_S \times \prod'_{p \notin S} {\rm O}(L_p)$. For $L' \subset V$ in the genus of $L$, the set ${\rm Isom}(L_S,L'_S)$ 
of $\Z_S$-linear isometries $L_S \rightarrow L'_S$ is a principal homogeneous space
under ${\rm O}(L_S)$. Define a {\it $K_S$-level structure} on $L'$
as a $K_S$-orbit of elements in ${\rm Isom}(L_S,L'_S)$.
The identity $L_S \rightarrow L_S$ defines a canonical $K_S$-structure $\alpha_0$  on $L_S$ 
and the map $g \mapsto (g(L), \alpha_0 \circ g^{-1})$ trivially is 
a bijection between ${\rm Gen}(K)$ and the set of all pairs
$(L',\alpha)$ with $L' \in {\rm Gen}_\Q(L)$ and $\alpha$ a $K_S$-level structure on $L'$.\ps

}
\end{example}

Note that a compact open subgroup of ${\rm O}(V_S)$ always fixes some $\Z_S$-lattice $L_S$ in $V_S$, hence is included in ${\rm O}(L_S)$. It follows that for any given $K$, 
we may always find $L \subset V$ and $S$ such that $K$ (hence ${\rm Gen}(K)$) is of the form studied in Example \eqref{gengenuslevel}.
A familiar example is the following. 

\begin{example} {\rm (}Principal level structures{\rm)}
{\rm
Fix a lattice $L \subset V$ and $N\geq 1$ an integer with $L_p$ unimodular in $V_p$ for each $p$ dividing $N$. Set $K = \prod_p K_p$ with $K_p \,=\, {\rm ker}\,( {\rm O}(L_p) \rightarrow {\rm O}(L_p \otimes \Z/N)\,)$
for all $p$ (hence $K_p={\rm O}(L_p)$ for $p \nmid N$). Then ${\rm Gen}(K)$ naturally identifies with the
set of pairs $(L',\alpha)$ with $L' \in {\rm Gen}_\Q(L)$ and $\alpha : L  \otimes \Z/N\Z \isomo L' \otimes \Z/N\Z$ an isomorphism  of quadratic spaces over $\Z/N\Z$
(use that ${\rm O}(L_p) \rightarrow {\rm O}(L \otimes \Z/p)$ is surjective for $p \nmid N$).}
\end{example}

\subsection{Spinor genera of $K$ and their mass}

We fix $V$ and $K$ as in \S \ref{classsets}.
By the discussion above, together with the finiteness of isometry classes of lattices in a given genus,
${\rm X}(K)$ is a finite set. 
By Lemmas \ref{spinornormOlemma} \& \ref{integralspin},  $\mu({\rm Spin}_V(\AAA_f))$ is a normal subgroup of ${\rm O}_V(\AAA_f)$
with abelian quotient. This shows that ${\rm O}(V)\mu({\rm Spin}_V(\AAA_f))K$ is a normal subgroup 
of ${\rm O}_V(\AAA_f)$ as well, with quotient group
\begin{equation} \label{spinorquot} {\rm S}(K) = {\rm O}(V)\backslash {\rm O}_V(\AAA_f)/\mu({\rm Spin}_V(\AAA_f))K.\end{equation}
We have a natural projection 
\begin{equation} 
\label{projspin} {\rm s}: {\rm X}(K) \rightarrow {\rm S}(K).
\end{equation}
For $x \in {\rm X}(K)$, its class ${\rm s}(x)$ in ${\rm S}(K)$ is called the {\it spinor genus} of $x$; these classes partition ${\rm Gen}(K)$ into the {\it spinor genera} of $K$. 
The group ${\rm S}(K)$ of spinor genera of $K$ is a finite elementary abelian $2$-group, 
which is dual to the group denoted $\Sigma(K)$ in \S \ref{prelimspin}.
We have a natural isomorphism ${\rm S}(K) \simeq \Delta/ \iota( {\rm O}(V)\, K )$ deduced from \eqref{idquotosp}, 
hence a natural exact sequence \begin{equation} 
\label{exactS12}
1 \rightarrow {\rm S}_1(K) \rightarrow {\rm S}(K) \rightarrow {\rm S}_2(K) \rightarrow 1,
\end{equation}
defined by this isomorphism and \eqref{exactDelta}.
\ps

In the case $V_\infty$ is isotropic of dimension $>2$, the strong approximation theorem shows 
that 
${\rm s}: {\rm X}(K) \rightarrow {\rm S}(K)$ is bijective (Kneser).
As our main results trivially hold in this case, we assume now
 $V_\infty$ is positive definite. 
 In this case ${\rm O}(V_\infty)$ is compact, so 
 ${\rm O}(V)$ is a discrete subgroup of ${\rm O}_V(\AAA_f)$.
 We choose the Haar measure on the unimodular locally 
 compact group ${\rm O}_V(\AAA_f)$ such that $K$ has measure $1$.
 We then equip ${\rm O}(V)$ with the counting measure, which endows ${\rm O}(V) \backslash {\rm O}_V(\AAA_f)$ 
with a finite measure denoted by ${\rm m}$. For $x \in {\rm O}_V(\AAA_f)$, we denote by $\Gamma_x$ the finite group ${\rm O}(V) \cap x K x^{-1}$. The measure of the subset ${\rm O}(V)xK$
of ${\rm O}(V) \backslash {\rm O}_V(\AAA_f)$ is then
$${\rm m}({\rm O}(V) x K ) = \frac{1}{|\Gamma_x|}.$$
It only depends on the class of $x$ in ${\rm X}(K)$, so it makes sense to write $|\Gamma_x|$ for $x \in {\rm X}(K)$. We have then
\begin{equation}
\label{defmk} 
{\rm m}_K:={\rm m}({\rm O}(V) \backslash {\rm O}_V(\AAA_f)) = \sum_{x \in {\rm X}(K)} \frac{1}{|\Gamma_x|}.
\end{equation} 
The following lemma is presumably well-known.

\begin{lemma} 
\label{massgenspin}
The spinor genera of $K$ all have the same measure $\frac{{\rm m}_K}{|{\rm S}(K)|}$.
\end{lemma}

\begin{pf} The right multiplication by $y \in {\rm O}_V(\AAA_f)$ preserves the Haar measure, commutes with left-multiplication by ${\rm O}(V)$, and sends the spinor genus ${\rm O}(V) x\, \mu({\rm Spin}_V(\AAA_f)) K$ to ${\rm O}(V) xy  \,\mu({\rm Spin}_V(\AAA_f)) K$. 
\end{pf}


\subsection{Statement of the main theorem}
\label{sect:statemainthm}
We fix $V, K$ and $S$ as in \S \ref{classsets}, as well as an arbitrary lattice $L \subset V$.
Up to enlarging $S$ we may and do assume furthermore (see \S \ref{sect:pnei})
\begin{equation} \label{ramKS} K_p={\rm O}(L_p)\,\,\,\, {\rm and}\,\,\,\, L_p \in \mathcal{U}(V_p)\,\,
\textrm{for all}\, \, p \notin S.
\end{equation}
Fix $p \notin S$. We identify ${\rm Gen}(K_p)$ with the set ${\rm Gen}(L_p)$ of all $\Z_p$-lattices in $V_p$ isometric to $L_p$, hence with the set $\mathcal{U}(V_p)$ of unimodular $\Z_p$-lattices of $V_p$ (Lemma \ref{orbunimloc}).
Setting ${\rm Gen}(K^p)={\rm Gen}(K_S) \times \prod'_{l \notin S \cup \{p\}} {\rm Gen}(K_l)$ we also have
\begin{equation} \label{decatp} {\rm Gen}(K) = {\rm Gen}(K_p) \times {\rm Gen}(K^p).\end{equation}
by \eqref{locdecGen}, and we shall write $x=(x_p,x^p)$ the corresponding decomposition of $x \in {\rm Gen}(K)$.
Two elements $x,y \in {\rm Gen}(K)$ are called {\it $p$-neighbors} if 
their $p$-th components $x_p$ and $y_p$, viewed as
elements of $\mathcal{U}(V_p)$, are $p$-neighbors.

\begin{definition}\label{defpneighbgen} For $x,y \in {\rm Gen}(K)$, we denote by ${\rm N}_p(x,y)$ the number of $p$-neighbors of $x$ which are in the same class as $y$ in ${\rm X}(K)$. 
\end{definition}

A few remarks are in order:\ps\ps

(R1) If $y$ is a $p$-neighbor of $x$, then $\sigma y$ is a $p$-neighbor of $\sigma x$ for all $\sigma \in {\rm O}(V)$. The quantity ${\rm N}_p(x,y)$ thus only depends on the classes of $x$ and of $y$ in ${\rm X}(K)$,
and we also write ${\rm N}_p(x,y)$ for $x,y \in {\rm X}(K)$. \ps\ps

(R2) Assume $\dim V \geq 3$. By Corollary~\ref{corsnpnei}, 
if $y$ is a $p$-neighbor of $x$ then the spinor genera ${\rm s}(y)$ and ${\rm s}(x)$ satisfy ${\rm s}(y) \, =\, {\rm \delta}_p\, {\rm s}(x)$, where $\delta_p \in {\rm S}(K)$ is defined by $\delta_p={\rm s}(g)$ for any element $g$ in ${\rm SO}_V(\AAA_f)$ with $g_l=1$ for $l\neq p$ and 
${\rm sn}(g_p) \in p \Z_p^\times$ (note $\Z_p^\times \subset {\rm sn}(K_p)$ for $p \notin S$ by Lemma \ref{integralspin}). We have $\delta_p \in {\rm S}_1(K)$, so ${\rm s}(x) \equiv {\rm s}(y) \bmod {\rm S}_1(K)$. \ps\ps

(R3) For $p \notin S$, the total number of $p$-neighbors of any $x$ in ${\rm Gen}(K)$ is $|{\rm C}_L(\Z/p)|$ by Lemma \ref{bijlinemap}, a quantity concretely given by Formula \eqref{cardCLp} and that we now denote by ${\rm c}_V(p)$.\ps\ps
 
\begin{thm}\label{thmmain} Assume $V$ is a quadratic space over $\Q$ with $\dim V >2$ and $V_\infty$ definite.
Let $K \subset {\rm O}_V(\AAA_f)$ be a compact open subgroup.
For all $x,y \in X(K)$ with ${\rm s}(x)\equiv {\rm s}(y) \bmod {\rm S}_1(K)$, we have 
$$\frac{{\rm N}_p(x,y)}{{\rm c}_V(p)} = \frac{1/|\Gamma_y|}{{\rm m}_K/|{\rm S}(K)|}\, \,+\,\, {\rm O}(\,\frac{1}{\sqrt{p}}\,)$$
when the prime $p$ goes to $\infty$ with the property ${\rm s}(y)\,=\,\delta_p\, {\rm s}(x)$.
We can even replace the $\frac{1}{\sqrt{p}}$ above by $\frac{1}{p}$ in the case $\dim V>4$.
\end{thm}
\ps\ps

\noindent By Lemma~\ref{massgenspin}, ${\rm m}_K/|{\rm S}(K)|$ is also the mass of the spinor genus of $y$ in ${\rm Gen}(K)$.

\begin{remark}
\label{mainthmlatticecase}
{\rm
Theorem~\ref{thmstatgen} is the special case of Theorem \ref{thmmain} in which $K$ has the form of Example~\ref{gengenus}
and $|{\rm S}(K)|=1$ (any $L' \in {\rm Gen}(L)$ is isometric to a lattice in ${\rm Gen}_\Q(L)$ by the Hasse-Minkowski theorem).
Note that when $K$ is as in Example~\ref{gengenus}, we always have $\det K = \{ \pm 1\}^P$, i.e. ${\rm S}_2(K)=1$, so in this case 
we have ${\rm s}(x) \equiv {\rm s}(y) \bmod {\rm S}_1(K)$ for all $x,y \in {\rm X}(K)$.
}
\end{remark}
\ps\ps

\begin{remark}
\label{spinordirichlet}
{\rm 
By definitions, for $p \notin S$ then $\delta_p$ is the image of the element $p \in (\Z/{\rm N}_S)^\times$ under the morphism $\delta_S$ 
in Formula \eqref{deltaT} and the natural isomorphism $\Delta_1(K) \isomo {\rm S}_1(K)$. 
In particular, $\delta_p \in {\rm S}_1(K)$ only depends on the class of $p$ in $(\Z/{\rm N}_S)^\times$.
For any $a \in {\rm S}_1(K)$, the primes $p \notin S$ with $\delta_p=a$ form thus a union of arithmetic 
progressions modulo ${\rm N}_S$, and there are infinitely many such primes by Dirichlet
(their Dirichlet density is $\frac{1}{|{\rm S}_1(K)|}$).
}
\end{remark}

\section{Proof of the main theorem}\label{sect:pf}

Fix $V$ and $K$ as in the statement of Theorem \ref{thmmain}. Define ${\rm M}(K)$ as the $\R$-vector space of maps ${\rm X}(K) \rightarrow \R$. This is an Euclidean space for 
\begin{equation}\label{petersson} \langle f, f' \rangle = \int_{{\rm O}(V)\backslash {\rm O}_V(\AAA_f)} \,f\, f'\, {\rm m}\,= \sum_{x \in {\rm X}(K)} \frac{1}{|\Gamma_x|} f(x) f'(x), \end{equation}
where ${\rm m}$ is the measure on ${\rm O}(V) \backslash {\rm O}(\AAA_f)$ recalled in \S \ref{classsets} ({\it Petersson inner product}). It follows that the elements $\sigma \in \Sigma(K)$, which are elements of ${\rm M}(K)$ by definitions, are pairwise orthogonal and with norm $\langle \sigma, \sigma \rangle = {\rm m}_K$.  We may thus write 
${\rm M}(K)  \,=\, (\oplus_{\sigma \in \Sigma(K)}\, \R \sigma\,)  \,\perp\, {\rm M}(K)^0$,  which defines ${\rm M}(K)^0$. 

\begin{lemma} 
\label{projsk}
For $\sigma \in \Sigma(K)$, the orthogonal projection ${\rm p}_\sigma : {\rm M}(K) \rightarrow \R \sigma$ is given by
${\rm p}_\sigma(f) =  \frac{1}{{\rm m}_K} (\sum_{x \in {\rm X}(K)}\frac{ f(x) \sigma(x)}{|\Gamma_x|})\, \, \sigma$.
\end{lemma}
\begin{pf} For $f \in {\rm M}(K)$ we have ${\rm p}_\sigma(f) = \lambda \sigma$ with 
$\lambda = \frac{\langle f, \sigma \rangle } {\langle \sigma , \sigma \rangle}$.
\end{pf}

Choose a lattice $L \subset V$ and a finite set $S$ of primes 
as in \S \ref{classsets} and satisfying \eqref{ramKS}. 
Fix $p \notin S$. Any $f \in {\rm M}(K)$ can be viewed  by \eqref{decatp} as a function 
${\rm Gen}(K^p) \times \Z[\mathcal{U}(V_p)] \rightarrow \R$ which is $\Z$-linear in the second coordinate.
We thus get a natural action on Hecke ring ${\rm H}_{V_p}$ on ${\rm M}(K)$.
For instance, we have 
\begin{equation}\label{defTpxy} {\rm T}_p(f)(x) = \sum_{y \in {\rm X}(K)} {\rm N}_p(x,y) f(y)\end{equation}
for all $x \in {\rm X}(K)$. 
The property $T^{\rm t}=T$ for all $T$ in ${\rm H}_{V_p}$ implies that each such operator of 
${\rm M}(K)$ is self-adjoint (see {\it e.g.} Lemma 4.2.3 in \cite{chlannes}).
Still by the decomposition \eqref{locdecGen}, and by the commutativity of each ${\rm H}_{V_p}$,
all those endomorphisms of ${\rm M}(K)$ commute.
It follows that they are are simultaneously diagonalizable.
Remarks (R2) and (R3) of \S \ref{sect:statemainthm}, together with Formula \ref{defTpxy}, imply:

\begin{lemma} \label{evtps} Each $\sigma \in \Sigma(K)$ is an eigenvector of ${\rm T}_p$, with eigenvalue ${\rm c}_V(p)\sigma(\delta_p)$.
\end{lemma}

\noindent In particular, ${\rm T}_p$ preserves ${\rm M}(K)^0$ for all $p \notin S$. 

\begin{thm}\label{eqmainthm} When the prime $p$ goes to infinity,  we have in ${\rm End}({\rm M}(K)^0)$
\begin{equation}\label{mainstat} \frac{ {\rm T}_p }{{\rm c}_V(p)} = {\rm O}(\frac{1}{\sqrt{p}}),\end{equation} 
and we can replace the $\frac{1}{\sqrt{p}}$ above by $\frac{1}{p}$ in the case $\dim V>4$.
\end{thm}

\begin{pf} (Theorem \ref{eqmainthm} implies Theorem \ref{thmmain}) By Lemma~\ref{evtps}
we have $\frac{{\rm T}_p}{{\rm c}_V(p)} \,= \, (\sum_{\sigma \in \Sigma(K)} \sigma(\delta_p) {\rm p}_\sigma) \perp u_p$
where $u_p$ denotes the restriction of $\frac{{\rm T}_p}{{\rm c}_V(p)}$ to ${\rm M}(K)^0$. Fix $x,y \in {\rm X}(K)$ and set $f=1_y$, the characteristic function of $y$. We have ${\rm T}_p(1_y)(x)= {\rm N}_p(x,y)$ by \eqref{defTpxy}. 
On the other hand, by Lemma \ref{projsk} we have 
$$\sum_{\sigma \in \Sigma(K)} \sigma(\delta_p) {\rm p}_\sigma(1_y)(x) = 
\sum_{\sigma \in \Sigma(K)}  \sigma(\delta_p) \frac{1}{{\rm m}_K} \frac{\sigma(y)}{|\Gamma_y|} \sigma(x),$$
which reduces to $\frac{1/|\Gamma_y|}{{\rm m}_K/|\Sigma(K)}$ assuming ${\rm s}(y)={\rm s}(x)\delta_p$, and we are done.
\end{pf}

It remains to prove Theorem \ref{eqmainthm}.
This step is not elementary, and we will give two different proofs it.
The first one will rely in a crucial way on
Arthur's classification of automorphic representations
of classical groups \cite{arthur}. 
Our first aim now is to state the statement we will need from this theory, 
in the form which is the most suitable for our purpose.
We need a few ingredients, denoted by (I1) to (I4) below.
\ps

(I1) {\it Cuspidal automorphic representations of ${\rm GL}_n$ over $\Q$}. We refer to \cite{boreljacquet} for this notion. For any such automorphic representation $\pi$, we have local components $\pi_v$ for $v=\infty$ or $v$ a prime $p$. 

-- The component $\pi_\infty$ is an irreducible Harish-Chandra module for ${\rm GL}_n(\R)$. As such, it has an infinitesimal character, that we may view following Langlands as semisimple conjugacy class ${\rm c}(\pi_\infty)$ in ${\rm M}_n(\C)$.\ps

-- For all but finitely many primes $p$, the representation $\pi_p$ is unramified, {\it i.e.} has nonzero invariants under ${\rm GL}_n(\Z_p)$. For any such $p$, the Satake isomorphism shows that $\pi_p$ is uniquely determined by its Satake parameter, that we may view following Langlands as semisimple conjugacy class ${\rm c}(\pi_p)$ in ${\rm GL}_n(\C)$.\ps

For any finite set of primes $S$ we denote by $\mathcal{X}^S({\rm GL}_n)$ the set of all collections $(x_v)_{v \notin S}$, such that $x_\infty$ is a semisimple conjugacy class in ${\rm M}_n(\C)$, and $x_p$ is a semisimple conjugacy class in ${\rm GL}_n(\C)$ for each prime $p \notin S$. By the discussion above, for any cuspidal automorphic representation $\pi$ of ${\rm GL}_n$, and for any $S$ such that $\pi_p$ is unramified for $p \notin S$, there is a well-defined element 
$${\rm c}(\pi) \in \mathcal{X}^S({\rm GL}_n),\,\,{\rm with}\,\, \, \, {\rm c}(\pi)_v:={\rm c}(\pi_v)\, \, \forall v \notin S.$$

(I2) {\it Operations on $\mathcal{X}^S({\rm GL}_n)$}. The direct sum and tensor product of conjugacy classes induce natural operations $\mathcal{X}^S({\rm GL}_n) \times \mathcal{X}^S({\rm GL}_m) \rightarrow \mathcal{X}^S({\rm GL}_{n+m})$ and $\mathcal{X}^S({\rm GL}_n) \times \mathcal{X}^S({\rm GL}_{m}) \rightarrow \mathcal{X}^S({\rm GL}_{nm})$, denoted respectively by $(x,y) \mapsto x \oplus y$ and $x \otimes y$. They are commutative, associative and distributive in an obvious way.
Also, an important role will be played by the {\it Arthur element} 
$${\rm A} \in \mathcal{X}^S({\rm GL}_2)$$
such that the eigenvalues of ${\rm A}_\infty$ are $\pm \frac{1}{2}$, and that of ${\rm A}_p$ are the positive real numbers $p^{\pm 1/2}$. For each integer $d\geq 0$, we shall also need a symmetric power morphism ${\rm GL}_2 \rightarrow {\rm GL}_{d+1}$; it induces a map ${\rm Sym}^d : \mathcal{X}^S({\rm GL}_2) \rightarrow 
 \mathcal{X}^S({\rm GL}_{d+1})$. 
\ps
(I3) {\it Automorphic forms for ${\rm O}_V$ and their parameters}.
By definition, an automorphic form for ${\rm O}_V$ is a map $f: {\rm O}(V) \backslash {\rm O}_V(\AAA) \rightarrow \C$ which is right-invariant under some compact open subgroup of ${\rm O}_V(\AAA_f)$
and which generates a finite dimensional vector space under all right translations by the compact group ${\rm O}(V_\infty)$. They form a $\C[{\rm O}_V(\AAA)]$-module denoted $\mathcal{A}({\rm O}_V)$. 
This module is well-known to be semisimple and admissible.
In what follows, we are only interested in the subspace of elements of $\mathcal{A}({\rm O}_V)$ which are right ${\rm O}(V_\infty)$-invariants;  it is nothing else than $\bigcup_{K \subset {\rm O}_V(\AAA_f)} {\rm M}(K)$, where $K$ runs among the compact open subgroups of ${\rm O}_V(\AAA_f)$. We shall also need below a variant of these definitions with ${\rm O}$ replaced everywhere by ${\rm SO}$.\ps

As in \S \ref{sect:pnei}, consider the complex algebraic group $\widehat{{\rm O}_V}$ defined by 
$\widehat{{\rm O}_V}={\rm Sp}_{2r}(\C)$ 
if $\dim V = 2r+1$ is odd, and by $\widehat{{\rm O}_V}={\rm O}_{2r}(\C)$ if $\dim V=2r$ is even.
We define ${\rm n}_V:=2r$ and have a natural complex representation 
${\rm St} : \widehat{{\rm O}_V} \rightarrow {\rm GL}_{{\rm n}_V}(\C)$.
For any finite set $S$ of primes, define $\mathcal{X}^S({\rm O}_V)$ as the set of all collections $(x_v)_{v \notin S}$, with $x_\infty$ a semi-simple conjugacy class in the complex Lie algebra of $\widehat{{\rm O}_V}$, and $x_p$ a semi-simple conjugacy class in $\widehat{{\rm O}_V}$. Assuming $f \in {\rm M}(K) \otimes \C$ is a common eigenvector for ${\rm H}_{V_p}$ for all primes $p \not \in S$, we attach to $f$ a unique element ${\rm c}(f) \in \mathcal{X}^S({\rm O}_V)$ defined as follows:\ps\ps

-- ${\rm c}(f)_\infty$ is the semisimple conjugacy class with eigenvalues in ${\rm St}$ the $2r$ elements $\pm \frac{\dim V-2i}{2}$ with $i=1,\dots,r$ (infinitesimal character of the trivial representation of ${\rm O}(V_\infty)$, generated by $f$),\ps

-- for $p \notin S$, the $\C[{\rm O}(V_p)]$-module generated by $f$ is isomorphic to ${\rm U}(\chi_p)$ with $\chi_p : {\rm H}_{V_p} \rightarrow \C$ defined by the eigenvector $f$. We define ${\rm c}(f)_p$ as the conjugacy class ${\rm c}(\chi_p)$ attached to $\chi_p$
under the Satake isomorphism (see \S \ref{sect:pnei}).\ps\ps

(I4) {\it Orthogonal-symplectic alternative.} If $\pi$ is a cuspidal automorphic representation of ${\rm GL}_n$ over $\Q$, then so is its contragredient $\pi^\vee$. If $\pi$ is self-dual, that is isomorphic to $\pi^\vee$, then it is either {\it orthogonal} or {\it symplectic}, in the sense of Arthur \cite{arthur}. Define accordingly the sign ${\rm s}(\pi)$ of $\pi$ to be $1$ or $-1$. \ps

\par \noindent 

We are now able to state the following main result of Arthur's theory. 
It is a standard consequence of Theorem~\ref{thm:existence_psi} proved by Ta\"ibi in the appendix.

\begin{thm}\label{arthurthm} Let $f \in {\rm M}(K)$ be a common eigenvector for all the elements of ${\rm H}_{V_p}$ for all $p \notin S$. Then there is a unique collection $(n_i,\pi_i,d_i)_{i \in I}$, 
with integers $n_i,d_i \geq 1$ satisfying ${\rm n}_V=\sum_{i \in I} n_i d_i$, and $\pi_i$ a cuspidal automorphic representation of ${\rm GL}_{n_i}$ unramified outside $S$, such that 
\begin{equation}\label{arthrel} {\rm St}({\rm c}(f)) \,=\, \bigoplus_{i \in I} \,{\rm c}(\pi_i)\, \otimes {\rm Sym}^{d_i-1} {\rm A}.\end{equation}
Moreover,  $\pi_i$ is self-dual with sign ${\rm s}(\pi_i) =(-1)^{\dim V+d_i-1}$ for each $i \in I$.
\end{thm}

\begin{pf} 
Let $\pi \subset \mathcal{A}({\rm O}_V)$ be an irreducible constituent of the ${\rm O}_V(\AAA)$-module generated by $f$.
By definitions, we have $\pi_\infty \simeq 1$ and $\pi_p$ is unramified with Satake parameter ${\rm c}(f)_p$. 
The restriction of functions ${\rm res}: \mathcal{A}({\rm O}_V) \rightarrow \mathcal{A}({\rm SO}_V)$ 
is ${\rm SO}_V(\AAA)$-equivariant, and we have ${\rm res} \pi\, \neq 0$ 
since any nonzero $\varphi$ in $\mathcal{A}({\rm O}_V)$ has an ${\rm O}_V(\AAA)$-translate 
with $\varphi(1) \neq 0$. Choose an irreducible constituent $\pi' \subset {\rm res}\, \pi$.
For each place $v$, the local component $\pi'_v$ is a constituent of the restriction of $\pi_v$ to ${\rm SO}(V_v)$. 
So $\pi'_\infty \simeq 1$, and for a prime $p \notin S$ the $1$ or $2$ constituents of $(\pi_p)_{|{\rm SO}(V_p)}$ 
are ${\rm O}(L_p)$-outer conjugate and have a Satake parameter belonging to that of $\pi_p$ (see Remark~\ref{heckeSOV}).
We conclude by applying Theorem~\ref{thm:existence_psi} to $\pi'$.
\end{pf}

\begin{pf} (First proof of Theorem \ref{eqmainthm})
	Let $f \in {\rm M}(K) \otimes \C$ be a common eigenvector for all ${\rm H}_{V_p}$ with $p \notin S$. Denote by $\lambda(p)$ the eigenvalue of ${\rm T}_p$ on $f$. As we have ${\rm c}_V(p) \sim p^{\dim V-2}$ for $p \rightarrow \infty$, it is enough to show that either $f$ is a linear combination of $\sigma$ with $\sigma \in \Sigma(K)$, or we have 
\begin{equation} \label{asstoshow} |\lambda(p)|={\rm O} ( p^{ \dim V -3 + \delta } ) \, \, \, \, {\rm for}\, \, p \longrightarrow \infty \end{equation}
with $\delta=0$ for $\dim V \geq 5$, and $\delta = 1/2$ for $3 \leq \dim V \leq 4$. \ps
	Let $(n_i,\pi_i,d_i)$, $i \in I$, be the collection associated to $f$ by Theorem \ref{arthurthm}. 
Denote by $\lambda_i(p)$ the trace of ${\rm c}(\pi_i)_p$. 
By \eqref{arthrel} and Lemma \ref{stakeTp} we have thus
\begin{equation} 
\label{formulvparthur} 
p^{\frac{-\dim V}{2}+1}\lambda(p) \,=\, \, \sum_{i \in I}\, \lambda_i(p) \, \frac{p^{d_i/2}-p^{-d_i/2}}{p^{1/2}-p^{-1/2}}.
\end{equation}
By the Jacquet-Shalika estimates \cite[Cor. p. 515]{jasha1}, we have $|\lambda_i(p)| < \,n_i\, p^{1/2}$ for all $i \in I$ and all $p \notin S$.
If we set $d = {\rm Max} \{ d_i \, \, | \, \, i \in I\}$, we deduce 
\begin{equation} \label{weakram} |\lambda(p)|  = {\rm O} ( p^{ \frac{ \dim V - 2 + d}{2} } )\, \, \, \, \, {\rm for}\, \, p \rightarrow \infty.
\end{equation}
In particular, the bound \eqref{asstoshow} follows from \eqref{weakram} in the case $d \leq  \dim V -4$,
so we definitely assume $d \geq \dim V-3$. Set $J=\{ j \in I\, \, |\, \, d_j=d\}$.
If the {\it generalized Ramanujan conjecture} holds for each $j \in J$, 
we have $|\lambda_j(p)| \leq n_j$ for $p \notin S$ and $j \in J$, hence the
better bound \begin{equation} \label{strongram} |\lambda(p)|  = {\rm O} ( p^{ \frac{ \dim V + d - 3}{2} } )\, \, \, \, \, {\rm for}\, \, p \rightarrow \infty.
\end{equation}
In this case, \eqref{asstoshow} follows for $d=\dim V-3$ as well. \ps

We consider first the case $\dim V = 2r+1$ is odd, hence $2r-2 \leq d \leq 2r$. For $j \in J$ we have ${\rm s}(\pi_j)=(-1)^{d}$ and  $n_j \leq \frac{2r}{d}$.\ps

-- Case $d=\dim V-3=2r-2$ (so $r\geq 2$ and $\dim V \geq 5$). 
For $j \in J$ we have $n_j \leq 1+\frac{1}{r-1}$, so $n_j=1$ and 
$\pi_j$ is a Hecke character of order $\leq 2$, hence trivially satisfies the Ramanujan conjecture: we conclude by \eqref{strongram}.
\ps

-- Case $d=\dim V-2=2r-1$. For $j \in J$ we have $n_j \leq 1+ \frac{1}{2r-1}$ and $\pi_j$ symplectic (hence
$n_j \equiv 0 \bmod 2$), so $r=1$, $\dim V=3$, $d=1$, $n_j=2$ and $I=\{j\}$. But the eigenvalues of  
${\rm c}_\infty(\pi_j)$ are $\pm 1/2$ by \eqref{arthrel}, and this is well-known to force $\pi_j$ to be generated by a classical holomorphic modular form of weight $2$.
The Ramanujan conjecture for such a $\pi_j$ is known (and goes back to Eichler, Shimura and Weil), so we conclude by \eqref{strongram}.\ps

-- In the remaining case $d=2r$, we have $J=\{j\}=I$ and $n_j=1$.
It follows from \eqref{arthrel} that we have ${\rm c}(f)_p = \pm \Delta$, with $\Delta$ defined as before Lemma \ref{sataketrivialeta}. By this lemma, the ${\rm O}(V_p)$-module generated by $f$ is one dimensional and 
isomorphic to $1$ or $\eta$. In particular, the function $f$
is right-invariant under $\mu({\rm Spin}(V_p))$. By Lemma~\ref{strongapp} (strong approximation), the function $f$ 
factors through the abelian group ${\rm S}(K)$, hence is a linear combination of elements of $\Sigma(K)$, 
and we are done. \ps\ps

We now discuss the case $\dim V = 2r$ is even ($r\geq 2$), hence $2r-3 \leq d \leq 2r$, which is quite similar. 
For $j \in J$ we have ${\rm s}(\pi_j)=(-1)^{d-1}$ and  $n_j \leq \frac{2r}{d}$. The case $d=2r$ is impossible, 
as it would force $n_j=1$ and $\pi_j$ symplectic (hence $n_j$ even). \ps

-- Case $d=\dim V-3=2r-3$. There is nothing to prove by \eqref{weakram} for $\dim V=4$,
so we may assume $r\geq 3$. For $j \in J$ we have $n_j \leq 1+\frac{3}{2r-3} \leq 2$ and $\pi_j$ orthogonal.
If we have $n_j \leq 1$ (hence Ramanujan for $\pi_j$) for each $j \in J$, we conclude by \eqref{strongram}. 
Otherwise we have $r=3$, $d=3$, $I=J=\{j\}$ and $n_j=2$.
Following Arthur's definition, a self-dual orthogonal cuspidal $\pi$ of ${\rm GL}_2$ must have a non-trivial, hence quadratic, central character $\epsilon$, hence satisfies $\pi \simeq \pi \otimes \epsilon$ by the strong multiplicity one theorem.
This forces $\pi$ to be the automorphic 
induction from a unitary Hecke character of the quadratic field defined by $\epsilon$ by \cite{labesselanglands}. 
A unitary Hecke character, hence $\pi$, satisfies the Ramanujan conjecture, 
so we conclude in the case $\dim V=6$ as well.\ps

-- Case $d=\dim V-2 = 2r-2$. We have $n_j \leq 1+ \frac{1}{r-1}$ and $\pi_j$ symplectic, 
so $n_i$ even, for $j \in J$. This forces $r=2$, $\dim V=4$, $d=2$, $n_j=2$ and $I=\{j\}$. But the eigenvalues of  
${\rm c}_\infty(\pi_j)$ are then $\pm 1/2$, so we know that $\pi_j$ is generated by a classical modular form of weight $2$, 
hence satisfies Ramanujan by Eichler-Shimura: we conclude by \eqref{strongram}.\ps

-- In the remaining case $d=2r-1=\dim V-1$, we must have $J=\{j\}$, $n_j=1$, $I=\{i,j\}$ with $i\neq j$, and $n_i=1$.
So $\pi_i$ and $\pi_j$ are two Hecke characters $\chi_i$ and $\chi_j$ of $\AAA^\times$, with $\chi_i^2=\chi_j^2=1$.
We recognize again ${\rm c}(f)_p = \pm \Delta$, with $\Delta$ defined as in Lemma \ref{sataketrivialeta}, and that 
lemma and Lemma~\ref{strongapp} imply similarly as above that $f$ is a linear combination of elements of $\Sigma(K)$. 
\end{pf}

We now give a second proof of Theorem~\ref{eqmainthm} which does not rely on Theorem~\ref{arthurthm} or \cite{arthur}, but rather on \cite{oh}.

\begin{prop} 
\label{majoh}
Let $V_p$ be a non-degenerate quadratic space over $\Q_p$ of dimension $\geq 5$, $L_p$ a unimodular $\Z_p$-lattice in $V_p$, $U$ a unitary irreducible unramified $\C[{\rm O}(V_p)]$-module, and $\lambda \in \C$ the eigenvalue of ${\rm T}_p$ on the line $U^{{\rm O}(L_p)}$. If $\dim U>1$ then we have $|\lambda| \leq \, |{\rm C}_{L_p}(\Z/p)| \, \frac{1}{p}\,(\frac{2 p}{p+1})^2$.
\end{prop}

\begin{pf} 
Write $U^{{\rm O}(L_p)} = \C e$ with $\langle e, e \rangle =1$. 
Write $C\,=\,{\rm O}(L_p)\,c\,{\rm O}(L_p)$ the double coset of elements 
$g \in {\rm O}(V_p)$ such that $g(L_p)$ is a $p$-neighbor of $L_p$. 
We have $\langle ge,e \rangle = \langle ce,e \rangle $ for all $g  \in C$, 
so $\lambda \,=\,  |{\rm C}_{L_p}(\Z/p)|\, \langle ce,e \rangle $ for all $c \in C$.\ps
We will apply Thm. 1.1. of \cite{oh} to the restriction of $U$ to $G:={\rm SO}(V_p)$. 
Note that the subgroup of $G^+ \subset G$ defined
{\it loc. cit.} (generated by unipotents) is normal in ${\rm O}(V_p)$. 
As $V_p$ is isotropic, \cite[Lemma 12]{tamagawa} shows that $G^+$ is the commutator subgroup 
of ${\rm O}(V_p)$, {\it i.e.} the kernel of the spinor norm on $G$ by Lemma \ref{spinornormOlemma},
so ${\rm O}(V_p)/G^+$ is a finite abelian group. 
As $U$ is irreducible of dimension $>1$, this shows that $U$ has no nonzero $G^+$-invariant. Note also that $G$ is semisimple, quasi-simple, and of $\Q_p$-rank $\geq 2$, 
since we have $\dim V_p \geq 5$ and $V_p$ contains the unimodular lattice $L_p$. 
Since ${\rm SO}_{L_p}$ is reductive over $\Z_p$, $K:={\rm SO}(L_p)$ 
is also a hyperspecial compact subgroup of $G$ in the sense of Bruhat-Tits. \ps
Assume first we are in the split case: $\dim V_p=2r+1$ is odd, or $\dim V_p=2r$ is even and $(-1)^r\det V_p$ is a square in $\Q_p^\times$.
We may write $L_p = \bigoplus_{i=1}^r (\Z_p e_i \oplus \Z_p f_i) \perp L_0$ 
with $e_i.e_j=f_i.f_j=0$, $e_i.f_j=\delta_{i,j}$ and ${\rm rk}_{\Z_p} L_0 \leq 1$. 
The sub-group scheme $A$ of ${\rm SO}_{L_p}$ preserving each $\Z_p e_i$, 
$\Z_p f_j$, and acting trivially on $L_0$, is a split maximal $\Z_p$-torus, 
and the sub-group scheme $B$ preserving each $\sum_{j \leq i} \Z_p e_i$ is 
a Borel subgroup. 
Denote by ${\rm X}^\ast$ and ${\rm X}_\ast$ respectively 
the character and cocharacter groups of $A$. 
The root system of $A$ is reduced of 
type ${\rm B}_r$ and ${\rm D}_r$ respectively,
and if $\varepsilon_i \in {\rm X}^*$ denotes the character of $A$ on $\Z_p e_i$, 
it is given in Planches II and IV of \cite{bourbaki}.  \ps

In the remaining case we have $\dim V_p=2r+2$ even and $(-1)^{r+1}\det V_p$
 is not a square in $\Q_p^\times$. We may still define $e_i ,f_j$, $L_0$, $A$, $B$, $\varepsilon_i$  and $\varepsilon_i^*$ as above, the only difference being that 
 $L_0$ is anisotropic of rank $2$ over $\Z_p$, so that $A$ is not anymore 
 a maximal torus (but still a maximal split torus). The root system of $A$ is reduced of type ${\rm C}_r$ as already seen during the proof of Lemma~\ref{stakeTp}, and is given
 in Planche III of \cite{bourbaki}. 
If $A^+$ denotes the subgroup of $a \in A(\Q_p)$ such that $|\alpha(a)|\leq 1$ 
for all positive roots $\alpha$ relative to $B$, then we have $G=KA^+K$ 
in all cases \cite[\S 9]{satake}.\ps

In all cases let $\varepsilon_i^* \in {\rm X}_\ast$ denote the dual basis of $\varepsilon_i$ 
with respect to the natural pairing between ${\rm X}_\ast$ and ${\rm X}^\ast$.
By Formula \eqref{LNef}, $\varepsilon_1^*(p)$ belongs 
to the double coset $C$ introduced in the first paragraph above.
The set $\mathcal{S}=\{\varepsilon_1-\varepsilon_2,\varepsilon_1+\varepsilon_2\}$ 
(note $r\geq 2$) is a strongly orthogonal system of positive (non multipliable) 
roots in the sense of \cite{oh}, since $2\varepsilon_i$ is not a root. 
For both $\alpha \in \mathcal{S}$ we have $\alpha(\varepsilon_1^\ast(p))=p$.
Since we have $\C[K]e=\C e$, as well as 
$\Xi(p)=\frac{1}{\sqrt{p}}\frac{2p}{p+1}$ 
where $\Xi$ is the Harish-Chandra function of ${\rm PGL}_2(\Q_p)$
recalled on p. 134 of \cite{oh}, Theorem 1.1 of \cite{oh} reads
$ \langle \varepsilon_1^\ast(p) e, e \rangle \leq \frac{1}{p}\,(\frac{2 p}{p+1})^2$,
and we are done.
\end{pf}

\begin{remark} {\rm  
(same assumptions) If $s \subset \widehat{{\rm O}_{V_p}}$ denotes 
the Satake parameter of $U$, Proposition \ref{majoh} asserts
$|{\rm Trace} \,{\rm St}(s)|\, \leq \,p^{-\frac{\dim V_p}{2}} 
\,|{\rm C}_{L_p}(\Z/p)| \,(\frac{2 p}{p+1})^2$ by Lemma~\ref{stakeTp}.
}
\end{remark}

\begin{pf} (Second proof of Theorem~\ref{eqmainthm})
Let $\pi \subset \mathcal{A}({\rm O}_V)$ be an irreducible 
${\rm O}_V(\AAA)$-submodule with $\pi_\infty\simeq 1$ and $\pi^K \neq 0$.
Assume $\pi$ is orthogonal to all $\sigma \in \Sigma(K)$. 
For each $p$ in $S$ then $\pi_p$ is not one-dimensional.
Indeed, otherwise the whole of $\pi$ would be one-dimensional by 
Lemma~\ref{strongapp} (strong approximation) and Lemma~\ref{spinornormOlemma}, hence generated by some
$\sigma \in \Sigma(K)$, a contradiction.
As the ${\rm O}_V(\AAA)$-module $\mathcal{A}({\rm O}_V)$ is unitary (Petersson inner product), so are $\pi$ and the $\pi_p$. Theorem~\ref{eqmainthm} follows thus 
from Proposition~\ref{majoh} for $\dim V\geq 5$. 
We treat the cases $3 \leq \dim V\leq 4$ directly as follows. 
In both cases we first choose an irreducible constituent $\pi' \subset \mathcal{A}({\rm SO}_V)$ 
of ${\rm res}\, \pi$ as explained in the proof of Theorem~\ref{arthurthm}. 
Note that $\pi'_p$ is not $1$-dimensional for $p \notin S$, otherwise $\pi_p$ 
would be $1$-dimensional as well, as the commutator subgroups of ${\rm SO}(V_p)$ 
and ${\rm O}(V_p)$ both coincide with the kernel of the spinor norm on ${\rm SO}(V_p)$.
\ps
	Assume now $\dim V=3$. We may find a definite quaternion $\Q$-algebra $D$ such that the algebraic $\Q$-group ${\rm SO}_V$ is isomorphic to the quotient of $D^\times$ by its center ${\mathbb G}_m$ (see {\it e.g.} p.73  \cite[Chap. 8]{campinas}). Then $\pi'$ may be viewed as an automorphic representation
of $D^\times$ with trivial central character. Recall $\pi'$ is not one-dimensional and satisfies $\pi'_\infty \simeq 1$. By the Jacquet-Langlands correspondence \cite{JL}, $\pi'$ corresponds thus to a cuspidal automorphic representation $\varpi$ of ${\rm GL}_2$ over $\Q$ generated by a holomorphic modular form of weight $2$. So the Ramanujan conjecture holds for $\varpi$ (Eichler, Shimura, Weil), hence for $\pi'$, and thus $\pi$, as well. For $p \notin S$, the eigenvalue of $p^{-1/2}\, {\rm T}_p$ on $\pi_p^{{\rm O}(L_p)}$ is thus of norm $\leq 2$ (Lemma~\ref{stakeTp}), and we conclude as $|{\rm C}_{L_p}(\Z/p)|=1+p$. \ps 
	Assume finally $\dim V=4$. It will be convenient to first choose 
an irreducible automorphic representation $\pi''$ of the proper similitude algebraic $\Q$-group ${\rm GSO}_V$ 
such that $\pi'$ is isomorphic to a constituent of the restriction of $\pi''$ to ${\rm SO}_V(\AAA)$: this is possible by \cite[Prop. 3.1.4]{patrikis} or \cite[Prop. 1]{chres}.  Since we have ${\rm GSO}_V(\R)=\R_{>0} \cdot{\rm SO}_V(\R)$, we may assume up to twisting
$\pi''$ if necessary that the central character $\chi$ of $\pi''$ satisfies $(\chi_\infty)_{|\R_{>0}}=1$.
We have then 
$\chi_\infty=1$ and $\pi''_\infty\simeq 1$ since $\pi'_\infty\simeq 1$. There are two possibilities (see e.g. \cite[Chap. 9, Thm. 12]{campinas}):
\ps

(Case a) $\det V$ is not a square in $\Q^\times$. If $F$ denotes the real quadratic field $\Q(\sqrt{\det V})$, then we may find a totally definite quaternion $F$-algebra $D$ such that the algebraic $\Q$-group ${\rm GSO}_V$ is isomorphic to the quotient of ${\rm Res}\, D^\times \times \mathbb{G}_m$ by its central $\Q$-torus ${\rm Res}\,\mathbb{G}_{m}$ embedded as $\{ (x,n(x)) \}$. Here we view $D^\times$ as an algebraic group over $F$, ${\rm Res}$ denotes the Weil restriction from $F$ to $\Q$,
and ${\rm n}$ denotes the norm of the extension $F/\Q$. As a consequence,
$\pi''$ may be viewed as an external tensor product 
$\chi \boxtimes \rho$ where $\rho$ is
an irreducible automorphic representation of $D^\times$ 
whose inverse central character is $\chi \circ {\rm n}$. 
Since $\pi''$ is not $1$-dimensional, $\rho$ is not $1$-dimensional as well.
By the Jacquet-Langlands correspondence, $\rho$ corresponds thus 
to a cuspidal automorphic representation $\varpi$ of ${\rm GL}_2$ over $F$.
For the two archimedean places $v$ of $F$, we have $\rho_v\simeq 1$ so
$\varpi_v$ is the lowest weight 
discrete series of ${\rm PGL}_2(\R)$ (and $\varpi$ is generated by a weight $(2,2)$ holomorphic Hilbert modular form). Such a $\varpi$ is known to satisfy the Ramanujan conjecture: at almost all finite places $v$ of $F$, which is all we need, this is due to Brylinski-Labesse \cite[Thm. 3.4.6]{brla} (and this even holds at all places by a result of Blasius \cite{blasius}). It follows that for all big enough primes $p$, then $\pi''_p$, hence $\pi'_p$ and $\pi_p$, are all tempered. For such a $p \notin S$, the eigenvalue of $p^{-1}\, {\rm T}_p$ on $\pi_p^{{\rm O}(L_p)}$ is thus of norm $\leq 4$ (Lemma~\ref{stakeTp}).  The asymptotics $|{\rm C}_{L_p}(\Z/p)| \simeq p^2$ not only concludes the proof of Theorem~\ref{eqmainthm} in (Case a), but even shows Corollary~\ref{addendumdim4}. \ps

(Case b) $\det V$ is a square in $\Q^\times$. 
Then we may find a definite quaternion $\Q$-algebra $D$ such that 
the algebraic $\Q$-group ${\rm GSO}_V$ is isomorphic to 
the quotient of $D^\times \times D^\times$ by its diagonal center ${\mathbb G}_m$. 
Then $\pi''$ may be viewed as an external tensor product $\pi_1 \boxtimes \pi_2$, 
with $\pi_1$ and $\pi_2$ two irreducible automorphic representations 
of $D^\times$ with inverse central characters $\chi^{\pm 1}$, 
and trivial archimedean components. 
Consider again the Jacquet-Langlands correspondent $\varpi_i$ of $\pi_i$. 
As in the case $\dim V=3$, each $\pi_i$ is either $1$-dimensional or tempered. 
As $\det V$ is a square, the reduced Langlands dual group of ${\rm SO}_V$ 
and ${\rm GSO}_V$ are connected and respectively isomorphic to ${\rm SO}_4(\C)$
and ${\rm GSpin}_4(\C) \simeq 
\{ (a,b) \in {\rm GL}_2(\C) \times {\rm GL}_2(\C)\, |\, \det a \det b=1\}$.
The morphism $\eta : {\rm GSpin}_4(\C) \rightarrow {\rm SO}_4(\C)$
dual to the inclusion ${\rm SO}_V \subset {\rm GSO}_V$ 
satisfies ${\rm St} \circ \eta (a,b) \simeq a \otimes b^{-1}$.
For each prime $p$ 
such that $\pi_p$, $\pi'_p$ and $\pi''_p$ are unramified,
their respective Satake parameters ${\rm c}(\pi_p)$, ${\rm c}(\pi'_p)$ and ${\rm c}(\pi''_p)$ are related by
\begin{equation} 
\label{quasiarth}
{\rm St}({\rm c}(\pi_p) ) = {\rm St}({\rm c}(\pi'_p) )= 
{\rm St} \circ \eta({\rm c}(\pi''_p))= 
{\rm c}(\varpi_1)_p \otimes {\rm c}(\varpi_2^\vee)_p,
\end{equation}
as semisimple conjugacy classes in ${\rm GL}_4(\C)$.
Indeed, the first equality follows from Remark \ref{heckeSOV}, 
and the last two from the compatibility of the Satake isomorphism with 
surjective morphisms with central kernel \cite{satake} 
(in the last equality, with isomorphisms). 
Formula~\eqref{quasiarth} is a variant of Formula \eqref{arthrel} for $\pi$ 
that is equally useful for our purpose.
Indeed, if both $\varpi_i$ are tempered then so are $\pi''_p$, $\pi'_p$ and $\pi_p$ 
for $p$ big enough, and we conclude as above by Lemma~\ref{stakeTp}. 
Otherwise, one of the $\varpi_i$ has dimension $1$, say $\varpi_1$, 
and the other $\varpi_2$ is tempered (as $\pi''$ is not $1$-dimensional). 
Then $\varpi_1$ has the form $\mu \circ \det$ for some Hecke unitary character 
$\mu$ (since $\mu_\infty=1$), and we have 
${\rm c}(\varpi_1)_p={\rm A}_p \otimes \mu_p(p)$, 
and thus 
${\rm St}({\rm c}_p(\pi)) ={\rm A}_p \otimes {\rm c}( \mu \otimes \varpi_2^\vee)_p$ 
by \eqref{quasiarth}, for
all but finitely many primes $p$, and 
we conclude by Lemma~\ref{stakeTp}. \end{pf}
\ps
\noindent In the study of (Case a) above we have proved:

\begin{cor} \label{addendumdim4} 
The conclusions of Theorems \ref{eqmainthm} and \ref{thmmain} 
also hold with $\frac{1}{\sqrt{p}}$ replaced by $\frac{1}{p}$ in the case $\dim V=4$ and $\det V$ is not a square. 
\end{cor}

\section{Unimodular lattices containing a given saturated lattice}
\label{sect:biased}

The aim of this section is to prove Theorem~\ref{secmainthmintro} of the introduction, and its generalization Theorem~\ref{secmainthmodd}.
We use from now on, and for short, the terminology {\it lattice} for {\it Euclidean integral lattice}.
If $L$ is a free abelian group of finite type, and if $A \subset L$ is a subgroup, recall that $A$ is said {\it saturated}
in $L$ if the quotient group $L/A$ is torsion free. It is equivalent to ask that $A$ is a direct summand in $L$, or the equality $A=(A \otimes \Q) \cap L$ in $L \otimes \Q$.

\subsection{The groupoid $\mathcal{G}_m(A)$}\label{groupoidGMA}

Recall that a groupoid $\mathcal{G}$ is a category in which all morphisms are isomorphisms. We say that $\mathcal{G}$ 
is {\it finite} if it has finitely many isomorphism classes of objects, and if furthermore each objet $x$ has a finite automorphism group ${\rm Aut}(x)$ in $\mathcal{G}$. The {\it mass} of such a $\mathcal{G}$ is defined by ${\rm mass}\, \mathcal{G} = \sum_x 1/|{\rm Aut}(x)|$, where $x$ runs among the finite set of isomorphism classes of objects in $\mathcal{G}$ (of course, we have a group isomorphism ${\rm Aut}(x) \simeq {\rm Aut}(y)$ for $x \simeq y$ in $\mathcal{G}$). \ps
If $X$ is a set equipped with an action of a group $G$, we denote by $[X/G]$ the groupoid $\mathcal{G}$ with set of objects $X$ and with ${\rm Hom}_\mathcal{G}(x,y):=\{ g \in G, gx=y\}$. This groupoid is finite if and only if $X$ has finitely many $G$-orbits and finite stabilizers in $G$. 
This holds of course if $X$ and $G$ are finite, in which case the orbit-stabilizer formula shows 
\begin{equation} \label{orbitstabl} {\rm mass}\, [X/G] = |X|/|G|. \end{equation} \ps

Fix a lattice $A$ and an integer $m$. Recall that for any lattice $U$, ${\rm emb}(A,U)$ denotes
the set of isometric embeddings $e : A \rightarrow U$ with $e(A)$ saturated in $U$.
We define a groupoid $\mathcal{G}_m(A)$ as follows. Its objects are the 
pairs $(U,e)$, with $U$ a rank $m$ unimodular (integral, Euclidean) lattice and $e \in {\rm emb}(A,U)$, and
an isomorphism $(U,e) \rightarrow (U',e')$ is an isometry $g : U \isomo U'$ 
satisfying $g \circ e = e'$. 
For a given unimodular lattice $U$ of rank $m$, the set ${\rm emb}(A,U)$ has a natural action of ${\rm O}(U)$,
and we have a natural fully faithful functor $[ {\rm emb}(A,U)/{\rm O}(U) ] \rightarrow \mathcal{G}_m(A)$,  $e \mapsto (U,e)$, whose essential image is the full subcategory of $\mathcal{G}_m(A)$ whose objects have the form $(U',e')$ with $U' \simeq U$. As a consequence, we have an equivalence
\begin{equation} \label{decompGmA}  \mathcal{G}_m(A) \,\simeq \,\coprod_U \, [ {\rm emb}(A,U)/{\rm O}(U) ]  \end{equation}
where $U$ runs among representatives of the (finitely many) isomorphism classes of rank $m$ unimodular lattices $U$. 
So $\mathcal{G}_m(A)$ is finite and we have ${\rm mass}\, [ {\rm emb}(A,U)/{\rm O}(U) ] \,=\, |{\rm emb}(A,U)|/|{\rm O}(U)|$ by \eqref{orbitstabl}. 
In particular, if we denote by $\mathcal{G}_m^{\rm even}(A)$ the full subgroupoid of pairs $(U,e)$ with $U$ even, 
and if ${\rm m}_m^{\rm even}(A)$ is as in Theorem~\ref{secmainthmintro}, 
we have proved:

\begin{lemma} \label{masslemma} We have ${\rm mass} \,\mathcal{G}_m^{\rm even}(A) \,=\, {\rm m}_m^{\rm even}(A)$. 
\end{lemma} 

\subsection{The groupoid $\mathcal{H}_m(A)$ and review of unimodular glueing} \label{glueingpart}

We fix $A$ and $B$ two integral lattices in respective Euclidean spaces $V_A$ and $V_B$
and set $V=V_A \perp V_B$.
We are interested in the set ${\rm Glue}(A,B)$ of unimodular lattices $L \subset V$ containing $A \perp B$, 
and with $A$ is saturated in $L$.
There is a well-known description of these lattices in \cite{nikulin} that we now recall.\ps

We denote by ${\rm Isom}(- {\rm res}\,A, {\rm res}\, B)$ the (possibly empty) set of group isomorphisms 
$\sigma : {\rm res}\, A \rightarrow {\rm res}\, B$ satisfying $\sigma(x).\sigma(y)=-x.y$ for all $x,y \in {\rm res}\, B$. 
For $\sigma \in {\rm Isom}(-{\rm res}\, A, {\rm res}\,B)$ we define a subgroup of ${\rm res}\, A \perp {\rm res}\, B$ by 
$${\rm I}(\sigma) : = \{ x + \sigma(x)\, \, |\, \, x \in {\rm res}\, A \}.$$
We check at once that $\sigma \mapsto {\rm I}(\sigma)$ is a bijection between ${\rm Isom}(-{\rm res}\, A, {\rm res}\,B)$ and the set of 
(bilinear) Lagrangians of ${\rm res}\, A \perp \, {\rm res}\, B$ which are transversal to ${\rm res}\,A$.
In the case $A$ and $B$ are even lattices, we have ${\rm q}({\rm I}(\sigma))=0$ if and only if $\sigma$ is a {\it quadratic similitude}, {\it i.e.} satisfies ${\rm q}(\sigma(x))=-{\rm q}(x)$ for all $x \in  {\rm res}\, A$.
Let $\pi_{A,B} :  A^\sharp \perp B^\sharp \rightarrow {\rm res}\, A \,\perp \, {\rm res}\, B$ denote the canonical projection, 
and define $L(\sigma)=\pi_{A,B}^{-1}\,{\rm I}(\sigma)$ for $\sigma \in {\rm Isom}(-{\rm res}\, A, {\rm res}\,B)$.
By \S \ref{sect:notations} (iii) we have:

\begin{lemma} \label{bijlemmalag}
The map $\sigma \mapsto L(\sigma)$\,\, is a bijection between ${\rm Isom}(-{\rm res}\, A,  {\rm res}\,B)$ and ${\rm Glue}(A,B)$.
In this bijection, $L(\sigma)$ is even if and only if $A$ and $B$ are even and $\sigma$ is a quadratic similitude.
\end{lemma}

That being said, we will now let $B$ vary, so we just fix a lattice $A$, say with rank $a$, as well as an integer $m \geq a$. 
We consider the following groupoid $\mathcal{H}_m(A)$. Its objects are the pairs $(B,\sigma)$ with $B$ a lattice of rank 
$m-a$ and $\sigma \in {\rm Isom}(-{\rm res}\, A, {\rm res}\,B)$, and an isomorphism $(B,\sigma) \rightarrow (B',\sigma')$ is an isometry $h : B \isomo B'$ verifying  ${\rm res}\, h \, \circ\,\sigma\, =\, \sigma' $. We have two natural functors 
\begin{equation} G : \mathcal{H}_m(A) \rightarrow \mathcal{G}_m(A)\, \, \, {\rm and}\, \, \, H : \mathcal{G}_m(A) \rightarrow \mathcal{H}_m(A).\end{equation}

-- The functor $G$ sends the object $(B,\sigma)$ to $(L(\sigma),e)$, 
where $e$ is the composition of the natural inclusions $A \subset A \perp B \subset L(\sigma)$, 
and the morphism $(B,\sigma) \rightarrow (B',\sigma')$ defined by $h : B \rightarrow B'$ to the morphism $G(h):={\rm id}_A \times h$. \ps

-- The functor $H$ sends the object $(L,e)$ to $(B,\sigma)$ with $B\,=\,L \,\cap \,e(A)^\perp$,
so that $L \in {\rm Glue}(e(A),B)$ has the form $L(\tau)$ for a unique element $\tau \in {\rm Isom}(-{\rm res}\, e(A),  {\rm res}\, B)$, 
and we set $\sigma = \tau\, \circ \, {\rm res}\, e$. 
Also, $H$ sends the morphism $(L,e) \rightarrow (L',e')$ defined by $g: L \rightarrow L'$ to $g_{|B}: B \rightarrow B'$, with $B'= L \,\cap \,e(A)^\perp$. \ps

It is straightforward to check that 
$G$ and $H$ are well-defined functors 
and that we have $H \circ G = {\rm id}_{\mathcal{H}_m(A)}$ and $G \circ H \simeq {\rm id}_{\mathcal{G}_m(A)}$:

\begin{lemma} \label{eqfunct} The functors $G$ and $H$ are inverse equivalences of groupoids.
\end{lemma}

We say that two objects $(B,\sigma)$ and $(B',\sigma')$ in $\mathcal{H}_m(A)$ have the {\it same parity}, 
if the lattices $B$ and $B'$ have the same parity and, in the case they are both even, if furthermore the isomorphism 
$\sigma' \circ \sigma^{-1} : {\rm res}\, B \isomo {\rm res}\, B'$ is an isomorphism of quadratic spaces (by definitions, it is an isomorphism of bilinear spaces). \ps

For any isometry $g : B \rightarrow B'$ with $B$ and $B'$ even, the isomorphism ${\rm res}\, g$ is a quadratic isometry, 
so two isomorphic objects of $\mathcal{H}_m(A)$ have the same parity. The following lemma follows from Lemma~\ref{bijlemmalag}.

\begin{lemma}\label{exerciseparity}  Let $(B,\sigma)$ and $(B',\sigma')$ be in $\mathcal{H}_m(A)$ and set $G(B,\sigma)=(U,e)$ and $G(B',\sigma')=(U',e')$. If $(B,\sigma)$ and $(B',\sigma')$ have the same parity, then $U$ and $U'$ have the same parity.
Conversely, if $U$ and $U'$ are even, then $(B,\sigma)$ and $(B',\sigma')$ have the same parity.
\end{lemma}

\begin{remark}
{\rm 
In the case $U$ and $U'$ above are both odd, $B$ and $B'$ may have a different parity.
For instance, if ${\rm I}_n$ denotes the standard unimodular lattice $\Z^n$, 
we may take $A={\rm I}_1$, $B={\rm I}_8$ and for $B'$ the ${\rm E}_8$ lattice.
Also, assuming $B$ and $B'$ are both even, $(B,\sigma)$ and $(B',\sigma)$ 
may have a different parity. For instance, if we set $A=B=B'$ all equal to the root lattice ${\rm D}_8$ then 
the isometry group of the quadratic space ${\rm res}\, B \simeq - {\rm res}\, B$ is $\Z/2$, 
and that of the underlying symmetric bilinear space is ${\rm SL}_2(\Z/2) \simeq {\rm S}_3$. 
There are thus $4$ elements $\sigma \in {\rm Isom}(-{\rm res}\, B, {\rm res}\, B)$ with ${\rm q} \circ \sigma \neq - {\rm q}$, 
and two distinct ${\rm O}(B)$-orbits of such elements.
}
\end{remark}

%
%

\subsection{Main statement}\label{subsect:mainstatement}

We can now state a version of Theorem \ref{secmainthmintro} that applies to odd lattices as well.
Let $L$ be a unimodular lattice of rank $m$ and $A \subset L$ a subgroup.
Denote by $A^{\rm sat} = L \cap (A \otimes \Q)$ the saturation of $A$ in $L$. 
Then $|A^{\rm sat}/A|$ divides $\det A$. 
It follows that for a prime $p \nmid \det A$, and $N$ a $p$-neighbor of $L$, 
we have $N \supset A$ if and only if $N \supset A^{\rm sat}$. 
As the prime $p$ will soon go to infinity, we may and do assume that $A=A^{\rm sat}$ is saturated in $L$. \ps

We denote by $B$ the orthogonal complement of $A$ in $L$. As $A$ is saturated in the unimodular lattice $L$, 
we have $\det B = \det A$ and this integer is also the index of $A \perp B$ in $L$ (see \S \ref{glueingpart}). We have thus \begin{equation} \label{orthomodp} L \otimes \Z/p \,=\, A \otimes \Z/p\, \perp\, B \otimes \Z/p\, \, \, \, {\rm for}\, \, p \nmid \det A,\end{equation}
and an inclusion ${\rm C}_B(\Z/p\Z) \subset {\rm C}_L(\Z/p\Z)$, since $B$ is saturated in $L$.
Recall $\mathcal{N}_p^A(L)$ denotes the set of $p$-neighbors of $L$ containing $A$.

\begin{lemma}\label{linemapA} For an odd prime $p$ not dividing $\det A$, the line map \eqref{defell} induces a bijection $\mathcal{N}_p^A(L) \isomo {\rm C}_B(\Z/p)$. 
\end{lemma}

\begin{pf} Let $N$ be a $p$-neighbor of $L$ and set $M=N\cap L$. We obviously have $N \supset A$ if and only if $M \supset A$. But we have $M =\{ v \in L \, \, |\, , \ell(N).v \equiv 0 \bmod p\}$, so $M \supset A$ if and only if $\ell(N) \in  (A \otimes \Z/p)^\perp$. We conclude by \eqref{orthomodp}.
\end{pf}

We have a tautological embedding ${\rm e}_L : A \rightarrow L$ given by natural inclusion.
For any $N \in \mathcal{N}_p^A(L)$, we also have a tautological embedding 
${\rm e}_N : A \rightarrow N$, again given by the natural inclusion.  
We even have ${\rm e}_N(A)=A$ saturated in $N$ whenever $p \nmid \det A$.
As we shall soon see (Lemma \ref{idenei}), $(N,{\rm e}_N)$ always has the same parity as $(L,{\rm e}_L)$.
For any $(L',e')$ in  $\mathcal{G}_m(A)$ of same parity as $(L,{\rm e}_L)$ we are thus interested in the number
${\rm N}_p^{A}(L,L',e')$ of $p$-neighbors $N \in \mathcal{N}_p^A(L)$ such that 
$(N,{\rm e}_N)$ is isomorphic to $(L',e')$ in $\mathcal{G}_m(A)$.\ps

For $\mathcal{C}=\mathcal{G}_m(A)$ or $\mathcal{H}_m(A)$, and for an object $\tau$ of $\mathcal{C}$, 
we denote by $\mathcal{C}^\tau$ the full subgroupoid of 
$\mathcal{C}$ whose objects have the same parity as $\tau$. 

\begin{thm} \label{secmainthmodd} Let $L$ be a unimodular lattice of rank $m$, $A$ a saturated 
subgroup of $L$, and set $B = L \cap A^\perp$.
Assume ${\rm rank}\, B \,\geq 3$ and that the inertial genus of $B$ is a single spinor genus.
For any $(L',e')$ in $\mathcal{G}_m(A)$ of same parity as $\tau:=(L,{\rm e}_L)$,
and for $p \rightarrow \infty$, we have 
\begin{equation} \label{biasedstat} \frac{{\rm N}_p^A(L,L',e')}{|{\rm C}_B(\Z/p)|} = \frac{1/|{\rm Aut}(L',e')|}{{\rm mass} \,\,\mathcal{G}_m^{\tau}(A)} \,+\, {\rm O}(\frac{1}{\sqrt{p}}),
\end{equation}
and we can replace the $\frac{1}{\sqrt{p}}$ above by $\frac{1}{p}$ in the case $\dim B \geq 5$.
\end{thm}

It remains to explain the condition on the ``inertial genus'' of $B$.
Consider the quadratic space $V=B \otimes \Q$ over $\Q$ and define a compact open subgroup $K \subset {\rm O}_V(\AAA)$ by  
$K \,=\, \prod_p K_p$ and $K_p={\rm I}(B_p)={\rm ker}\,( {\rm O}(B_p) \rightarrow {\rm O}({\rm res}\, B_p))$ (see \ref{sect:notations} \S (iv)).
We call ${\rm Gen}(K)$ the {\it inertial genus} of $B$.
The assumption on ${\rm Gen}(K)$ in Theorem~\ref{secmainthmodd} is $|\Sigma(K)|=1$. Recall ${\rm g}(X)$
denotes the minimal number of generators of the finite abelian group $X$.

\begin{lemma}\label{inertialspin} Let $C$ be an integral $\Z_p$-lattice of rank $r$ in a non-degenerate quadratic space over $\Q_p$, and set 
$g={\rm g}({\rm res}\, C)$. 
Assume $g \leq r-2$ if $p$ is odd or if $p=2$ and $C$ is even, and $g \leq r-3$ otherwise.
Then ${\rm sn}( {\rm I}(C)\cap {\rm SO}(C))$ contains the image of $\Z_p^\times$ in $\Q_p^\times/\Q_p^{\times,2}$, and we have $\det {\rm I}(C)=\{\pm 1\}$.
\end{lemma}

\begin{pf} We may write $C = U \perp U'$ for some integral $\Z_p$-submodules $U$ and $U'$ satisfying
$\det U  \in \Z_p^\times$ and $U'.U' \subset p \Z_p$ (use e.g. \cite[Prop. 1.8.1]{nikulin}).
We have thus ${\rm res}\, U=0$, ${\rm res}\, C = {\rm res}\, U'$, ${\rm rank}\, U'=g$ and ${\rm rank}\, U = r-g \geq 2$.
The subgroup ${\rm O}(U) \times 1 \subset {\rm O}(C)$ acts trivially on ${\rm res}\, C$.
As we have ${\rm rank}\, U \geq 1$, this proves $\det {\rm I}(C) \supset \det {\rm O}(U) = \{\pm 1\}$.
To conclude, it is enough to show $\Z_p^\times \Q_p^{\times,2} \subset {\rm sn}( {\rm SO}(U))$.
If $p$ is odd, or if $p=2$ and $U$ has a rank $2$ even (unimodular) summand, Lemma~\ref{integralspin}
implies ${\rm q}(U)\supset \Z_p^\times$, and we are done. In the remaining case, we have $p=2$ and $U \simeq \langle a\rangle \perp \langle b\rangle  \perp \langle c \rangle \perp D$ with $a,b,c \in \{\pm1, \pm 3\}$. A straightforward computation shows then $\mathcal{U}\mathcal{U}=\Z_2^\times$ where $\mathcal{U}=\{ x.x \,|\, x \in U\} \cap \Z_2^\times$, and we are done again.
\end{pf}

\begin{cor}\label{inequimpliesspin} Let $L$ be a unimodular lattice of rank $m$, $A$ a saturated 
subgroup of $L$ and set $B=L \cap A^\perp$. Assume either $B$ is even and
${\rm rank}\, A + {\rm g}({\rm res}\, A) \leq m-2$, or $B$ is odd and ${\rm rank}\, A + {\rm g}({\rm res}\, A) \leq m-3$.
Then the inertial genus of $B$ is a single spinor genus and ${\rm rank}\, B \geq 3$.
\end{cor}

\begin{pf} We have a group isomorphism ${\rm res}\, A \simeq {\rm res}\, B$ by Lemma~\ref{bijlemmalag}.
The assertion $|\Sigma(K)|=1$ follows then from Lemma~\ref{inertialspin} and the second assertion of Lemma~\ref{finitenessSK}, using $K'_p=K_p$ for all $p$ (and $T=\emptyset$). The trivial inequality 
${\rm g}({\rm res}\, B) \geq 0$
shows  ${\rm rank}\, B \geq 3$, or $B$ is even, ${\rm rank}\, B=2$ and ${\rm res}\, B=0$.
But there is no even unimodular lattice of rank $2$.
\end{pf}

\begin{pf}  (Theorem~\ref{secmainthmodd} implies Theorem~\ref{secmainthmintro})
As $L$ is even, we have $\mathcal{G}_m^\tau(A)=\mathcal{G}_m^{\rm even}(A)$ by Lemma~\ref{exerciseparity}.
The mass of this groupoid is ${\rm m}_n^{\rm even}(A)$ (Lemma~\ref{masslemma}).
We conclude by Corollary~\ref{inequimpliesspin} and Formulas~\ref{orbitstabl} \& \ref{decompGmA}.
\end{pf}

\subsection{Proof of Theorem \ref{secmainthmodd}}

Fix $L$, $A$, $B$ as in the statement, and set $W=B \otimes \Q$ and $K=\prod_p K_p$ with $K_p={\rm I}(B_p)$ as above.
We have $H(L,{\rm e}_L)=(B,\sigma)$ for a unique $\sigma \in {\rm Isom}(-{\rm res}\, A, {\rm res}\, B)$.\ps

For $g \in {\rm O}_W(\AAA_f)$, we have a lattice $g(B) \in {\rm Gen}_\Q(B)$ (see Example \S \ref{gengenus}), as well as an element 
$\sigma_g \in {\rm Isom}(-{\rm res}\, A, {\rm res}\, g(B))$ defined by $\sigma_g\, =\, {\rm res}\, g \,\circ \,\sigma$.
The meaning of ${\rm res}\, g$ here uses the canonical decomposition of ${\rm res}\, B$ as an orthogonal sum of all the ${\rm res}\, B_p$ (see \S \ref{sect:notations} (iv)). The pair $(g(B),\sigma_g)$ only depends on $gK \in {\rm Gen}(K)$ by definition of $K$, and defines an object in $\mathcal{H}_m(A)$.
We have $\sigma_g \,\circ\, \sigma^{-1} \,=\, {\rm res}\, g: {\rm res}\, B \rightarrow {\rm res}\, g(B)$ so $(g(B),\sigma_g)$ has the same type as $(B,\sigma)$. Also, any $\gamma \in {\rm O}(W)$ trivially induces for all $g \in {\rm O}_W(\AAA_f)$ a unique morphism $(g(B),\sigma_g) \rightarrow (\gamma g (B),\sigma_{\gamma g})$. We have thus defined a functor
\begin{equation}\label{functorphi} \Phi : [{\rm Gen}(K)/{\rm O}(W)] \rightarrow \mathcal{H}_m^{\tau}(A), \, \, \, {\rm with}\,\, \, \tau:=(B,\sigma).\end{equation}

\begin{lemma} \label{equivphi} The functor \eqref{functorphi} is an equivalence of groupoids. 
\end{lemma}

\begin{pf} The full faithfulness of $\Phi$ is trivial, so we focus on its essential surjectivity. 
Fix an object $(B',\sigma')$ in $\mathcal{H}_m^\tau(A)$.
By definition, the positive definite lattices $B$ and $B'$ have the same parity, 
the same determinant $|{\rm res}\, A|$, isometric 
bilinear residues (and even isometric quadratic residues if they are even, by the parity condition). 
By \cite[Cor. 1.16.3]{nikulin}, this implies that they are in the same genus.
By Hasse-Minkowski, we may thus assume $B' \subset W=B \otimes \Q$, 
and then $B'=g(B)$ for some $g \in {\rm O}_W(\AAA)$. 
It is enough to show that up to replacing $g$ by $gk$ for some $k $ in the stabilizer 
$\prod_p {\rm O}(L_p)$ of $B$ in ${\rm O}_W(\AAA)$, we may achieve $\sigma_g=\sigma'$.
As $\sigma_{gk} \circ \sigma^{-1}= {\rm res}\, g \circ {\rm res}\, k$,  
we conclude by Theorems 1.9.5 and 1.16.4 in \cite{nikulin}: the morphism 
${\rm res}: {\rm O}(B_p) \rightarrow {\rm O}({\rm res}\, B_p)$ is surjective for $p$ odd 
or for $p=2$ and $B_2$ odd, and its image is the whole subgroup of quadratic isometries for $p=2$ and $B_2$ even.
\end{pf}

As we have just seen, ${\rm Gen}(K)$ is naturally identified with the set of all $(B',\sigma')$ with $B' \subset W$ in the genus of $B$,
and $(B',\sigma')$ an object in $\mathcal{H}_m^\tau(A)$. Let $S$ be the set of odd primes not dividing $\det A=\det B$.
For $p \notin S$, and according to this identification and Definition \ref{defpneighbgen}, the $p$-neighbors of 
$(B,\sigma)$ are the $(B', \sigma')$ where $B'$ is a $p$-neighbor of $B$ and $\sigma'$ is the composition of $\sigma$ and of the 
natural (``identity'') isomorphism ${\rm res}\, B \isomo {\rm res}\, B'$ deduced from the equality $B'[1/p]=B[1/p]$. 
In particular, $\sigma'$ is uniquely determined by $B'$, which in turn is uniquely determined by its line $l:=\ell(B')$
in ${\rm C}_B(\Z/p)$. We may thus write $B(l)$ the lattice $B'$, and $\sigma(l)$ the isometry $\sigma'$. On the other hand, the set ${\rm C}_B(\Z/p)$ is also the orthogonal 
of $A\otimes \Z/p$ in ${\rm C}_L(\Z/p)$ by \eqref{orthomodp}, and it parameterizes the $p$-neighbors of $L$ 
containing $A$. We denote by $L(l)$ this $p$-neighbor of $L$ defined by $l$. 
Theorem \ref{secmainthmodd} follows then from Theorem \ref{thmmain} applied to $W$ and $K$, 
Lemmas \ref{eqfunct} \& \ref{equivphi}, and Lemma \ref{idenei} below.

\begin{lemma} \label{idenei} Let $p$ be an odd prime not dividing $\det A$.  For any $l$ in ${\rm C}_B(\Z/p)$ we have $H(L(l),{\rm e}_{L(l)})=(B(l),\sigma(l))$.
\end{lemma}

\begin{pf} Recall we have $H(L,{\rm e}_L)=(B,\sigma)$. 
Choose $g \in {\rm O}_W(\AAA_f)$ with $gB=B(l)$ and $g_q=1$ for $q \neq p$;
by definition we have ${\rm res}\, g \circ \sigma = \sigma(l)$. 
Set $V=L\otimes \Q$ and define $h \in {\rm O}_V(\AAA_f)$ by $h= {\rm id}_{A \otimes \AAA_f} \perp g$.
We have $hL \supset A$. We claim $hL=L(l)$. 
Indeed, $L(l)$ is the unique unimodular lattice $\neq L$ with $L(l)[1/p]=L[1/p]$
and $L(l)_p \supset {\rm M}_p(L_p,l)$. But we clearly have $(hL)[1/p]=L[1/p]$, and using 
$L_p=A_p\perp B_p$ and  ${\rm M}_p(L_p,l) = A_p \perp {\rm M}_p(B_p,l)$, 
we deduce $(hL)_p = A_p \perp B(l)_p$, so $hL \neq L$ and 
$(hL)_p \supset {\rm M}_p(L_p,l)$, which proves the claim. 
Define now $\sigma'$ by $H(L(l),{\rm e}_{L(l)})=(B(l),\sigma')$.
We have ${\rm I}(\sigma')=L(l)/(A \perp B(l))=hL/(A \perp gB)={\rm I}({\rm res}\,g \circ \sigma) $, and thus $\sigma'={\rm res}\,g \circ \sigma=\sigma(l)$.
\end{pf}


\appendix
\titleformat{\section}{\bfseries}{\appendixname~\thesection : }{0.5em}{}
\section[Arthur parameters for definite orthogonal groups]{Arthur parameters for automorphic representations of orthogonal groups}
\begin{center} Olivier Ta\"ibi \end{center}

\makeatletter
\renewcommand\thethmapp{\Alph{section}.\arabic{thmapp}}
\makeatother

\makeatletter
\renewcommand\thelemmapp{\Alph{section}.\arabic{lemmapp}}
\makeatother
\makeatletter
\renewcommand\thepropapp{\Alph{section}.\arabic{propapp}}
\makeatother
\makeatletter
\renewcommand\theremapp{\Alph{section}.\arabic{remapp}}
\makeatother
\makeatletter
\renewcommand\theequation{\Alph{section}.\arabic{equation}}
\@addtoreset{equation}{section}
\makeatother

We explain how to deduce Theorem \ref{thm:existence_psi} below from Arthur's
results in \cite{arthur} using the stabilization of the trace formula.
The proof is a piece of the argument in \cite{taibi_cpctmult}, save for
Proposition \ref{prop:Sdisc_discrete} (see Remark \ref{rem:diffcpctmult}), so we will
omit a few details.
We borrow some notations and definitions introduced in \S \ref{sect:notations}
and \S \ref{prelimspin}.\ps

Let $(V,q)$ be a quadratic space over $\Q$.
Assume that $V_\infty$ is definite.
Let $G = \mathrm{SO}_V$ be the associated special orthogonal group, considered
as an algebraic group over $\Q$.
Let $S$ be a finite set of places of $\Q$ containing the Archimedean place and
all the finite places $p$ such that $G_{\Qp}$ is ramified.
Let $m \in \Z_{\geq 1}$ be the product of all prime numbers in $S$.
There exists a reductive model $\underline{G}$ of $G$ over $\Z[m^{-1}]$.
Such a model is concretely obtained as follows.
We may choose a $\Z$-lattice $L$ in $V$ such that for any prime number $p$ not
in $S$ there exists $n \in \Z$ such that the lattice $(L_p, p^n q)$ is
unimodular if $p>2$ or $\dim V$ is even, or integral with determinant in $2
\Z_2^\times$ if $p=2$ and $\dim V$ is odd.
The subgroup scheme $G \subset \mathrm{O}_{L[1/m]}$ defined as the kernel of the
determinant for $\dim V$ odd, and as the kernel of the Dickson determinant
otherwise, is a reductive model $\underline{G}$ of $G$ over $\Z[m^{-1}]$
satisfying $\underline{G}(\Zp) = \mathrm{SO}(L_p)$ for any $p\not\in S$.
An irreducible representation $\pi = \bigotimes'_v \pi_v$  of $G(\AAA)$ is said
to be unramified away from $S$, if we have $\pi_p^{\underline{G}(\Zp)} \neq 0$
for all primes $p \notin S$, for some choice of reductive model $\underline{G}$
as above.\ps

If $\dim V$ is even choose for every place $v$ of $\Q$ an element $\delta_v$ of
$\mathrm{O}(V_v)$ having determinant $-1$ and order $2$, which belongs to
$\mathrm{O}(L_v)$ if $v$ does not belong to $S$.
Let then $\theta_v$ be the involutive automorphism of $G_{\Q_v}$ defined as
conjugation by $\delta_v$.
To treat both cases uniformly we simply let $\theta_v$ be the identity
automorphism of $G_{\Q_v}$ if $\dim V$ is odd. \ps

Recall that the Langlands dual group ${}^L G$ is by definition a semi-direct
product $\widehat{G} \rtimes \mathrm{W}_\Q$ where $\widehat{G}$ is a connected
reductive group over $\C$ and $\mathrm{W}_\Q$ is the Weil group of $\Q$.
Also recall that $G$ can be realized as a pure inner form of a quasi-split group
$G^*$, also realized as a special orthogonal group, and that we have a canonical
identification between ${}^L G$ and ${}^L G^*$.
The automorphisms $\widehat{\theta_v}$ of $\widehat{G} \rtimes
\mathrm{W}_{\Q_v}$ are induced by the same automorphism $\widehat{\theta}$ of
${}^L G$. \ps

Recall from \cite[\S 1.4]{arthur} that Arthur introduced a set
$\widetilde{\Psi}(G^*)$ of substitutes for global Arthur-Langlands parameters for
$G^*$, and define $\widetilde{\Psi}(G)=\widetilde{\Psi}(G^*)$.
Such a substitute is a formal (unordered) sum
\[ \psi = \bigoplus_{i \in I_\psi} \psi_i^{\oplus \ell_i} \oplus \bigoplus_{j
\in J_\psi} \left( \psi_j^{\oplus \ell_j} \oplus (\psi_j^\vee)^{\oplus \ell_j}
\right). \]
In this expression each $\psi_i = (\pi_i, d_i)$ is a pair where $\pi_i$ is a
unitary cuspidal automorphic representation of a general linear group
(considered up to isomorphism) and $d_i \geq 1$ is an integer.
We think of such a pair as the tensor product of the putative global Langlands
parameter of $\pi_i$ tensored with the irreducible algebraic representation of
$\SL_2$ of dimension $d_i$, and we simply denote $\psi_i =\pi_i[d_i]$.
For $i \in I_\psi \sqcup J_\psi$ we denote $\psi_i^\vee = \pi_i^\vee[d_i]$ where
$\pi_i^\vee$ is the contragredient representation of $\pi_i$.
We have $\psi_i^\vee = \psi_i$ if and only if $i \in I_\psi$, $\ell_i \geq 1$
for all $i \in I_\psi \sqcup J_\psi$, and the factors $(\psi_i)_{i \in I_\psi}$,
$(\psi_j)_{j \in J_\psi}$ and $(\psi_j^\vee)_{j \in J_\psi}$ are all
distinct.\ps

To any $\psi$ in $\widetilde{\Psi}(G)$ Arthur associated an extension
$\mathcal{L}_\psi$ of the Weil group $\mathrm{W}_{\Q}$ of $\Q$ by a product over
$i \in I_\psi \sqcup J_\psi$ of symplectic or special orthogonal groups (for $i
\in I_\psi$) and general linear groups (for $i \in J_\psi$) over $\C$.
Fix a standard representation $\mathrm{Std}_G$ of ${}^L G$ as in \cite[\S
2.1]{taibi_cpctmult}.
Arthur observed that we have a natural $\widehat{G}$-conjugacy class of
parameters $\widetilde{\psi}_G: \mathcal{L}_\psi \times \SL_2 \to {}^L G$
characterized by its composition with $\mathrm{Std}_G$, up to outer automorphism
in the even orthogonal case.
A parameter $\psi \in \widetilde{\Psi}(G)$ is called discrete if the centralizer
$S_\psi$ of $\widetilde{\psi}_G$ in $\widehat{G}$ is finite modulo
$Z(\widehat{G})^{\mathrm{Gal}(\overline{\Q}/\Q)}$.
It is elementary to check that $\psi$ is discrete if and only if one of the
following conditions is satisfied:
\begin{itemize}
\item
  the group $G$ is a split special orthogonal group in dimension $2$, i.e.\ $G$
  is isomorphic to $\mathrm{GL}_1$, or
\item
  the index set $J_\psi$ is empty and for any $i \in I_\psi$ we have $\ell_i =
  1$.
\end{itemize}
Denote $\widetilde{\Psi}_\mathrm{disc}(G) \subset \widetilde{\Psi}(G)$ the
subset of discrete parameters. \ps

For every $\psi \in \widetilde{\Psi}(G)$ and every place $v$ of $\Q$ we have,
thanks to the local Langlands correspondance for general linear groups, a
parameter $\mathrm{WD}_{\Q_v} \rightarrow \mathcal{L}_\psi$, where
$\mathrm{WD}_{\Q_v}$ denotes the Weil-Deligne group of $\Q_v$, which is
well-defined up to conjugation by $\mathcal{L}_\psi$ and outer conjugation on
the even orthogonal factors of $\mathcal{L}_\psi$.
Composing with $\widetilde{\psi}_G$ gives the localization $\psi_v:
\mathrm{WD}_{\Q_v} \times \SL_2 \rightarrow {}^L G$ of $\psi$ at $v$, again
uniquely determined up to conjugation by $\widehat{G} \rtimes
\{1,\widehat{\theta}\}$.
In particular for $v=\infty$ we associate to $\psi \in \widetilde{\Psi}(G)$ an
``infinitesimal character'' $\widetilde{\nu}(\psi)$ which is a
$\{1,\widehat{\theta}\}$-orbit of $\widehat{G}$-conjugacy classes of semisimple
elements in the Lie algebra of $\widehat{G}$, again determined by its image
under the standard representation.
Similarly, for any prime number $p$ such that $\psi$ is unramified at $p$ the
group $G_{\Qp}$ is an inner form of an unramified group and we have an
associated $\{1,\widehat{\theta}\}$-orbit $c_p(\psi)$ of $\widehat{G}$-conjugacy
classes of semisimple elements in $\widehat{G} \rtimes \mathrm{Frob}_p$.\ps

\begin{thmapp} \label{thm:existence_psi}
  As above let $G$ be a special orthogonal group over $\Q$ such that $G(\R)$ is
  compact.
  Let $S$ be a finite set of places of $\Q$ containing the Archimedean place and
  such that for any prime number $p \not\in S$ the group $G_{\Qp}$ is
  unramified, and let $\pi = \bigotimes'_v \pi_v$ be an automorphic
  representation of $G(\AAA)$ which is unramified away from $S$.
  There exists $\psi \in \widetilde{\Psi}_\mathrm{disc}(G)$ such that the
  following two conditions are satisfied.
  \begin{itemize}
  \item
    The infinitesimal character of $\pi_\infty$ belongs to $\widetilde{\nu}(\psi)$.
  \item
    For any prime number $p \not\in S$ the Satake parameter of $\pi_p$ belongs
    to $c_p(\psi)$.
  \end{itemize}
\end{thmapp}

The rest of this appendix is devoted to the proof of this theorem.
For $p \not\in S$ endow $G(\Qp)$ with the Haar measure giving
$\underline{G}(\Zp)$ volume $1$.
Let $\mathrm{H}^S_\mathrm{unr}(G)$ be the unramified Hecke algebra (with complex
coefficients) for the group $G$ at the finite places not in $S$.
It is naturally a restricted tensor product, over these places $p$, of the
unramified Hecke algebras $\mathrm{H}(G(\Qp), \underline{G}(\Zp))$, which can be
identified with the algebras $\mathrm{H}'_{V_p} \otimes \C$ of \S
\ref{sect:pnei}.
Let $\widetilde{\mathrm{H}}^S_\mathrm{unr}(G)$ be the subalgebra consisting of
functions invariant under $\theta_p$ for any prime $p \not\in S$.
It is also a restricted tensor product, over these places $p$, of the algebras
\[ \widetilde{\mathrm{H}}(G(\Qp), \underline{G}(\Zp)) := \mathrm{H}(G(\Qp),
\underline{G}(\Zp))^{\theta_p}, \]
which can be identified with the algebras $\mathrm{H}_{V_p} \otimes \C$ of \S
\ref{sect:pnei}.
For $v \in S$ fix a Haar measure on $G(\Q_v)$ and let $\mathrm{H}(G(\Q_v))$ be
the convolution algebra of smooth compactly supported functions on $G(\Q_v)$
which are moreover bi-$G(\R)$-finite in the Archimedean case.
Let $\mathrm{H}_S(G) = \bigotimes_{v \in S} \mathrm{H}(G(\Q_v))$.
Let $\widetilde{\mathrm{H}}_S(G)$ be the subalgebra of functions invariant under
$\theta_v$ for any $v \in S$.
It is also the tensor product of the subalgebras
$\widetilde{\mathrm{H}}(G(\Q_v)) := \mathrm{H}(G(\Q_v))^{\theta_v}$. \ps

Let $\pi = \bigotimes'_v \pi_v$ be an automorphic representation for $G$ which
is unramified away from $S$.
In particular $\pi_\infty$ is a continuous irreducible finite-dimensional
representation of $G(\R)$.
Let $\nu$ be its infinitesimal character, seen as a semi\-simple conjugacy class
in the Lie algebra of the dual group $\widehat{G}$.
This conjugacy class is regular.
We will only use the orbit $\widetilde{\nu} = \{ \nu,
\widehat{\theta_\infty}(\nu) \}$ of $\nu$.
Let $\mathfrak{z}$ be the center of the enveloping algebra of the complex Lie
algebra $\C \otimes_{\R} \operatorname{Lie} G(\R)$.
Let $\mathcal{A}(G, \widetilde{\nu})$ be the space of automorphic forms for $G$
which is the sum of the eigenspaces for the action of $\mathfrak{z}$
corresponding to the elements of $\widetilde{\nu}$.
Let $f \in \mathrm{H}^S_\mathrm{unr}(G) \otimes \mathrm{H}_S(G)$.
We will not lose any generality by assuming that $f$ decomposes as a product
$\prod_v f_v$.
The stabilization of the trace formula for $G$ (due to Arthur in great
generality), refined by infinitesimal character, reads
\begin{equation} \label{eq:STF}
  \tr \left( f \,\middle|\, \mathcal{A}(G, \widetilde{\nu}) \right) =
  \sum_{\mathfrak{e} = (H, \mathcal{H}, s, \xi)} \iota(\mathfrak{e})
  \sum_{\substack{\nu' \\ \xi(\nu') \in \widetilde{\nu}}} S^H_{\mathrm{disc,
  \nu'}}(f^H)
\end{equation}
where the first sum is over equivalence classes of elliptic endoscopic data
$\mathfrak{e}$ which are unramified away from $S$ (there are only finitely many
such equivalence classes), and the second sum is over the set of
$\widehat{H}$-conjugacy classes $\nu'$ of semisimple elements of the Lie algebra
of $\widehat{H}$ mapping to an element of $\widetilde{\nu}$.\ps

The stable linear forms $S^H_{\mathrm{disc}, \nu'}$ were defined inductively by
Arthur.
We do not need to recall their precise definitions as we will use Arthur's
expansion for these stable linear forms below.
To recall how the element $f^H$ of the ad\'elic Hecke algebra of $H$ is defined
we first need to fix an embedding ${}^L \xi: {}^L H \to {}^L G$ extending $\xi:
\widehat{H} \to \widehat{G}$ as recalled in \cite[\S 2.3]{taibi_cpctmult}.
In the cases considered in this appendix this embedding satisfies a simple
compatibility property: the quasi-split connected reductive group $H$ naturally
decomposes as a product $H_1 \times H_2$ where each $H_i$ is of the same type as
$G$, and the composition of ${}^L \xi$ with the standard representation of ${}^L
G$ is simply the direct sum of the standard representations of the ${}^L H_i$.
(The order on the factors $H_1$ and $H_2$ is not uniquely determined, and it can
happen for one of the two factors to be trivial.)
In particular this choice of ${}^L \xi$ has a formal analogue
\begin{align*}
  \Xi: \widetilde{\Psi}(H) & \longrightarrow \widetilde{\Psi}(G) \\
  (\psi'_1, \psi'_2) & \longmapsto \psi'_1 \oplus \psi'_2.
\end{align*}
The function $f^H$ is only determined through its stable orbital
integrals\footnote{implicitly a Haar measure on $H(\AAA)$ is fixed here},
and is defined to be a transfer of $f$.
The notion of transfer is defined unambiguously in this global setting, but to
write $f^H$ as a product of functions $f^H_v$ over all places $v$ of $\Q$ we
need to fix normalizations of transfer factors at all places.
For any $p \not\in S$ the embedding ${}^L \xi$ is unramified at $p$ and so there
is a distinguished normalization of transfer factors for the localization of
$\mathfrak{e}$ at $p$, which depends on ${}^L \xi$ and on the choice of
hyperspecial compact open subgroup $\underline{G}(\Zp)$ of $G(\Qp)$.
At these places the function $f^H_p$ can be taken to be unramified, i.e.\
bi-invariant under a hyperspecial compact open subgroup $K_{H,p}$ of $H(\Qp)$.
In fact the fundamental lemma says that the map $\mathrm{H}(G(\Qp),
\underline{G}(\Zp)) \to \mathrm{H}(H(\Qp), K_{H,p})$ dual (via the Satake
isomorphisms for $G$ and $H$) to the embedding ${}^L \xi$ realizes transfer.
(This holds true for any choice of hyperspecial subgroup $K_{H,p}$.)
For the purpose of this appendix it is enough to choose the normalizations of
transfer factors at the places in $S$ in any way which makes the product of the
local transfer factors equal to the canonical global transfer factor.
This is possible because $S$ is not empty. \ps

We now make use of Arthur's expansion for the stable linear forms
$S_{\mathrm{disc},\nu'}^H$.
For this it is crucial to observe a property of transfer factors: as explained
in \cite[\S 3.2.3]{taibi_cpctmult} if $\dim V$ is even then it is evident on
Waldspurger's explicit formulas that they are invariant under the simultaneous
action of $\theta_v$ on $G$ and of an outer non-inner action on one of the
$H_i$'s
\footnote{For this invariance property the choice of a particular normalization
of transfer factors is irrelevant.}.
As a consequence of this invariance property, if we assume that $f = \prod_v
f_v$ belongs to $\widetilde{\mathrm{H}}^S_\mathrm{unr}(G) \otimes
\widetilde{\mathrm{H}}_S(G)$ then the stable orbital integrals of each $f^H_v$
are invariant under the outer automorphism group of each $H_i$.
From now on we make this assumption on $f$.
Let $\widetilde{\Psi}(H, \widetilde{\nu})$ be the set of $\psi' \in
\widetilde{\Psi}(H)$ such that $\widetilde{\nu}(\Xi(\psi')) =
\xi(\widetilde{\nu}(\psi'))$ is equal to $\widetilde{\nu}$.

Arthur proved a stable multiplicity formula which applies to the inner sum on
the right-hand side of \eqref{eq:STF}.
By \cite[Corollary 3.4.2]{arthur} and the vanishing assertion in Theorem 4.1.2
loc.\ cit.\ we have the following expansion for the inner sum in \eqref{eq:STF}:
\begin{equation} \label{eq:Sdisc_coarse_exp}
  \sum_{\substack{\nu' \\ \xi(\nu') \in \widetilde{\nu}}} S_{\mathrm{disc},
  \nu'}^H(f^H) = \sum_{\psi' \in \widetilde{\Psi}(H, \widetilde{\nu})}
  S_{\mathrm{disc}, \psi'}^H(f^H).
\end{equation}
On the right-hand side only finitely many $\psi'$ have non-vanishing
contribution (\cite[Lemma 3.3.1]{arthur}, see also \S X.5 and XI.6 in
\cite{SFTT2}).
Moreover Theorem 4.1.2 loc.\ cit.\ provides a formula for each term
$S_{\mathrm{disc}, \psi'}^H(f^H)$.
It turns out that a substantial simplification occurs thanks to the fact that
$\nu$ is regular.
We prove this simplification in two steps.

\begin{lemmapp}
  Let $G'$ be a quasi-split symplectic or special orthogonal group over a number
  field $F$.
  Let $\psi \in \widetilde{\Psi}(G')$.
  Assume that there exists a place $v$ of $F$ such that the localization
  $\psi_v: \mathrm{WD}_{F_v} \times \SL_2 \to {}^L G'$ is regular, in the sense
  that the neutral connected component of its centralizer in $\widehat{G'}$ is
  contained in a maximal torus.
  Assume moreover that the factor $\sigma(\overline{S}_\psi^0)$ appearing in
  \textup{\cite[Theorem 4.1.2]{arthur}} is non-zero.
  Then the parameter $\psi$ is discrete.
\end{lemmapp}
\begin{proof}
  We may assume that $G'$ is not a split special orthogonal group in dimension
  $2$, because in this case every parameter is discrete.
  Under this assumption the group
  $Z(\widehat{G'})^{\mathrm{Gal}(\overline{\Q}/\Q)}$ is finite.

  The neutral connected component $S_\psi^0$ of the centralizer $S_\psi$ of
  $\widetilde{\psi}_{G'}$ in $\widehat{G'}$ decomposes naturally as a product of
  general linear, symplectic and special orthogonal groups over $\C$ (see
  \cite[(1.4.8)]{arthur}).
  The assumption $\sigma(\overline{S}_\psi^0) \neq 0$ implies that
  $\overline{S}_\psi^0$ has finite center, hence so does $S_\psi^0$ since
  $Z(\widehat{G'})^{\mathrm{Gal}(\overline{\Q}/\Q)}$ is finite, and so no
  general linear group and no special orthogonal in dimension $2$ appears in
  this product decomposition of $S_\psi^0$.
  The assumption on $\psi_v$ implies that $S_\psi^0$ is commutative, and so the
  symplectic and special orthogonal factors composing $S_\psi^0$ are all
  trivial.
  This implies that the complex reductive group $S_\psi$ is finite and so $\psi$
  is discrete.
\end{proof}

\begin{propapp} \label{prop:Sdisc_discrete}
  Let $\psi' = (\psi'_1, \psi'_2) \in \widetilde{\Psi}(H, \widetilde{\nu})$ be
  such that for each $i \in \{1,2\}$ the factor
  $\sigma(\overline{S}_{\psi'_i}^0)$ appearing in \textup{\cite[Theorem
  4.1.2]{arthur}} does not vanish.
  Then the parameter $\psi := \Xi(\psi') \in \widetilde{\Psi}(G,
  \widetilde{\nu})$ is discrete.
\end{propapp}
\begin{proof}
  The parameter $\widetilde{\psi}_G$ can be realized as ${}^L \xi \circ
  \widetilde{\psi'}_H$.
  Thanks to the fact that $\nu$ is regular we know that $S_\psi$ is contained in
  a maximal torus of $\widehat{G}$.
  The semisimple element $s$ of $\widehat{G}$ occurring in $\mathfrak{e}$
  belongs to $S_\psi$.
  In particular $S_\psi$ is contained in $\xi(\widehat{H})$ and so it is
  equal to $\xi(S_{\psi'})$.
  Thanks to the previous lemma applied to each factor of $\psi'$ we know that
  $S_{\psi'} / Z(\widehat{H})^{\mathrm{Gal}(\overline{\Q}/\Q)}$ is finite.
  By ellipticity of $\mathfrak{e}$ we also know that
  \[  \xi(Z(\widehat{H})^{\mathrm{Gal}(\overline{\Q}/\Q)})
  /Z(\widehat{G})^{\mathrm{Gal}(\overline{\Q}/\Q)} \]
  is finite.
  It follows that $S_\psi / Z(\widehat{G})^{\mathrm{Gal}(\overline{\Q}/\Q)}$ is
  finite.
\end{proof}

\begin{remapp}\label{rem:diffcpctmult}
  In \textup{\cite[\S 4]{taibi_cpctmult}} this property was deduced from the
  fact that the local parameter $\psi_\infty$ is discrete because it is
  Adams-Johnson, which uses the fact that $\nu$ is \emph{algebraic} regular.
  The above argument is more global in nature and seems more robust.
\end{remapp}

Let $\widetilde{\Psi}_{G-\mathrm{disc}}(H, \widetilde{\nu})$ be the set of
$\psi' \in \widetilde{\Psi}(H, \widetilde{\nu})$ such that $\Xi(\psi')$ is
discrete.
We clearly have $\widetilde{\Psi}_{G-\mathrm{disc}}(H, \widetilde{\nu}) \subset
\widetilde{\Psi}_{\mathrm{disc}}(H, \widetilde{\nu})$.
Combining \eqref{eq:STF}, \eqref{eq:Sdisc_coarse_exp}, \cite[Theorem
4.1.2]{arthur} and Proposition \ref{prop:Sdisc_discrete} we obtain
\begin{equation} \label{eq:STF_psi}
  \tr \left( f \,\middle|\, \mathcal{A}(G, \widetilde{\nu}) \right) =
  \sum_{\mathfrak{e} = (H, \mathcal{H}, s, \xi)} \iota(\mathfrak{e}) \sum_{\psi'
  \in \widetilde{\Psi}_{G-\mathrm{disc}}(H, \widetilde{\nu})} m_{\psi'}
  |\mathcal{S}_{\psi'}|^{-1} \Lambda_{\psi'}(f^H).
\end{equation}
where $\Lambda_{\psi'}$ is the stable linear form that Arthur associates to
$\psi'$, which decomposes as a product over all places $v$ of $\Q$ of stable
linear forms $\Lambda_{\psi'_v}$.
The precise definitions of the positive integers $m_{\psi'}$ and
$|\mathcal{S}_{\psi'}|$, defined as product over the factors $H_i$, do not
matter for the purpose of this appendix.
By the characterization of $\Lambda_{\psi'}$ using twisted endoscopy for general
linear groups, using the twisted fundamental lemma \cite{LFT1}, \cite{LFT2} and
the fundamental lemma we have
\[ \Lambda_{\psi'}(f^H) = \prod_{v \in S} \Lambda_{\psi'_v}(f^H_v) \times
\prod_{p \not\in S} \mathrm{Sat}(f_v)({}^L \xi(c(\psi'_v))) \]
if $\psi'$ is unramified away from $S$, and $\Lambda_{\psi'}(f^H) = 0$
otherwise.
Here $\mathrm{Sat}$ is the Satake isomorphism between the unramified Hecke
algebra $\mathrm{H}(G(\Q_p), \underline{G}(\Z_p))$ and the algebra of algebraic
functions on $\widehat{G} \rtimes \mathrm{Frob}_p$ which are invariant under
conjugation by $\widehat{G}$.
As a consequence the stabilization of the trace formula can be refined by
families of Satake parameters as follows.
Let $\widetilde{x}^S = (\widetilde{x}_p)_{p \not\in S}$ be a family of
$\{1,\widehat{\theta}\}$-orbits of semisimple $\widehat{G}$-conjugacy classes in
$\widehat{G} \rtimes \mathrm{Frob}_p$.
This family corresponds to a character of
$\widetilde{\mathrm{H}}^S_\mathrm{unr}(G)$ (see Remark \ref{heckeSOV}).
Let
\[ \mathcal{A}(G, \widetilde{\nu}, \widetilde{x}^S) \subset \mathcal{A}(G,
\widetilde{\nu})^{\prod_{p \not\in S} \underline{G}(\Zp)} \]
be the eigenspace for this character.
For any $f_S = \prod_{v \in S} f_v \in \widetilde{\mathrm{H}}_S(G)$ we have
\begin{align} \label{eq:STF_Satake}
  & \tr \left( f_S \,\middle|\, \mathcal{A}(G, \widetilde{\nu}, \widetilde{x}^S)
  \right) \\
  =\ & \sum_{\mathfrak{e} = (H, \mathcal{H}, s, \xi)} \iota(\mathfrak{e})
  \sum_{\psi' \in \widetilde{\Psi}_{G-\mathrm{disc}}(H, \widetilde{\nu},
  \widetilde{x}^S) } m_{\psi'} |\mathcal{S}_{\psi'}|^{-1} \prod_{v \in S}
  \Lambda_{\psi'_v}(f^H_v), \nonumber
\end{align}
where $\widetilde{\Psi}_{G-\mathrm{disc}}(H, \widetilde{\nu}, \widetilde{x}^S)$
is the set of $\psi' \in \widetilde{\Psi}_{G-\mathrm{disc}}(H, \widetilde{\nu})$
which are unramified away from $S$ and such that for any prime number $p \not\in
S$ the orbit $c_p(\Xi(\psi')) = {}^L \xi(c_p(\psi'))$ is equal to
$\widetilde{x}_p$.
Now for any prime number $p \not\in S$ take $\widetilde{x}_p$ equal to the
$\{1,\widehat{\theta}\}$-orbit of the Satake parameter of $\pi_p$.
We claim that $f_S$ can be chosen so that the left-hand side of
\eqref{eq:STF_Satake} is non-zero.
We can choose $f_\infty \in \widetilde{\mathrm{H}}(G(\R))$ such that for any
irreducible continuous (finite-dimensional) representation $\sigma$ of $G(\R)$
we have
\[ \tr \sigma(f_\infty) = \begin{cases}
  1 & \text{ if } \sigma \simeq \pi_\infty \text{ or } \sigma \simeq
    \pi_\infty^{\theta_\infty},\\
  0 & \text{ otherwise}.
\end{cases} \]
For $p$ a prime number in $S$ take $f_p \in \widetilde{\mathrm{H}}(G(\Qp),
\underline{G}(\Zp))$ to be the characteristic function of a compact open
subgroup $K_p$ of $G(\Qp)$ such that we have $\theta_p(K_p) = K_p$ and
$\pi_p^{K_p} \neq 0$.
With this choice of $f_S$ the left-hand side of \eqref{eq:STF_Satake} is a
positive integer.
In particular the double sum on the right-hand side is not empty, i.e.\ there
exists an elliptic endoscopic datum $\mathfrak{e} = (H, \mathcal{H}, s, \xi)$
for $G$ which is unramified away from $S$ and $\psi' \in
\widetilde{\Psi}_{G-\mathrm{disc}}(H, \widetilde{\nu}, \widetilde{x}^S)$.
Letting $\psi := \Xi (\psi')$ concludes the proof of Theorem
\ref{thm:existence_psi}.

\begin{remapp}
The same argument proves a more general statement:
$\Q$ could be replaced by an arbitrary number field $F$, $G$ could be any
inner form of a quasi-split symplectic or special orthogonal group over $F$,
and $\pi$ any discrete automorphic representation of $G$ such that for some Archimedean place $v$ of $F$ the
infinitesimal character of $\pi_v$ is regular (or a similar assumption at a
finite place).\ps
\end{remapp}

\end{document}